%% file: szlenk.tex
\documentclass{amsart}

\usepackage{amssymb,ifthen,amsthm}

\input{mypreamble.tex}

\newtheorem{thm}{Theorem}
\newtheorem{mainthm}{Theorem}

\newtheorem{lem}[thm]{Lemma}
\newtheorem{prop}[thm]{Proposition}
\newtheorem{cor}[thm]{Corollary}
\newtheorem*{problem}{Problem}

\theoremstyle{definition}

\newtheorem*{defn}{Definition}

\theoremstyle{remark}

\newtheorem*{rem}{Remark}

\newtheorem*{exs}{Examples}

\newcommand{\odd}{\ensuremath{T_\infty^{\text{even}}}}
\newcommand{\ms}{\ensuremath{\text{ms}}}
\newcommand{\sder}[1]{\ensuremath{^{(#1)}_S}} 
\newcommand{\si}{\ensuremath{\mathrm{I}_S}} 

\author{E. Odell, Th. Schlumprecht and A.~Zs\'ak}
\thanks{Research of the first two authors was supported by the
  National Science Foundation}
\title[bounded Szlenk index]{Banach spaces of bounded Szlenk index}
\subjclass[2000]{46B20, 54H05}
\keywords{Szlenk index, universal space, embedding into FDDs,
  Effros-Borel structure, analytic classes}

\begin{document}

\begin{abstract}
  For a countable ordinal $\alpha$ we denote by $\cC_\alpha$ the class
  of separable, reflexive Banach spaces whose Szlenk index and the
  Szlenk index of their dual are bounded by $\alpha$. We show that
  each $\cC_\alpha$ admits a separable, reflexive universal space. We
  also show that spaces in the class
  $\cC_{\omega^{\alpha\cdot\omega}}$ embed into spaces of
  the same class with a basis. As a consequence we deduce that
  each $\cC_\alpha$ is analytic in the Effros-Borel structure of
  subspaces of $C[0,1]$.
\end{abstract}

\maketitle

\section{Introduction}

A well known result that dates back to the early days of Banach space
theory~\cite[Th\'eor\`eme 9, page 185]{Banach:32} states that every
separable Banach space embeds into $C[0,1]$, \ie that $C[0,1]$ is
universal for the class of all separable Banach
spaces. Pe{\l}czy\'nski~\cite{Pelczynski:69} refined this result by
showing that there are Banach spaces $X$ with a basis and $X_u$ with
an unconditional basis such that every space with a basis or with an
unconditional basis is isomorphic to a complemented subspace of $X$ or
$X_u$, respectively.

By a famous result of Szlenk~\cite{Szlenk:68}, there is no separable
reflexive Banach space $X$ which contains isomorphically all separable
reflexive Banach spaces. Bourgain~\cite{Bourgain:80} sharpened this
result by showing that a separable Banach space which contains all
separable reflexive Banach spaces must contain $C[0,1]$ and, thus, all
separable Banach spaces. Szlenk proved his result by introducing for a
Banach space $X$ an ordinal-index $\sz(X)$ (see section 5) which is
countable if and only if $X$ has a separable dual and is hereditary
(if $Y$ embeds in $X$, then $\sz(Y)\kleq\sz(X)$). Moreover he showed
that for any countable ordinal $\alpha$ there is a separable reflexive
space $X$ for which $\sz(X)\kge\alpha$. Bourgain achieved his result
by introducing an index which measures how well finite sections of the
Schauder basis of $C[0,1]$ embed in $X$, and then he proved statements
analogous to Szlenk's approach.

Bourgain then raised the question whether there is a separable,
reflexive space that contains isomorphically all uniformly convex
Banach spaces. This problem was solved recently by the first two
authors. It was proven in~\cite{OS:06b} that if $X$ is a separable,
uniformly convex Banach space, then there exist $1\kle q\kleq
p\kle\infty$ and a reflexive space $Z$ with an FDD (finite-dimensional
decomposition) $(E_n)$ so that $X$ embeds into $Z$ and $(E_n)$
satisfies block $(\ell_p,\ell_q)$-estimates. This means that for some
$C$ we have
\[
C^{-1} \Big( \sum \norm{z_i}^p\Big)^{1/p} \leq \Bignorm{\sum z_i} \leq
C\Big( \sum \norm{z_i}^q\Big)^{1/q}
\]
for any block sequence $(z_i)$ of $(E_n)$. In fact this holds
(see~\cite[Theorems~3.4 and~4.1]{OS:06a}) if we merely know that
\[
X\in \cC_\omega = \big\{ Y:\,Y\text{ is separable, reflexive,
}\sz(Y)\kleq\omega,\ \sz(Y^*)\kleq\omega \big\}\ .
\]
Using then a result of S.~Pruss~\cite{Prus:83}, who solved Bourgain's
question within the class of Banach spaces with FDDs, we deduce that
there exists a universal reflexive space $Z$ for the class
$\cC_\omega$, and in fact $Z\kin \cC_{\omega^2}$.

Inspired by the results of~\cite{OS:06b}, A.~Pe{\l}czy\'nski  raised
the question whether a similar result could be proved for the
classes $\cC_\alpha$, where $\alpha\kle\omega_1$ and where
$\cC_\alpha$ is defined analogously to the class $\cC_\omega$. This
would mean that, although the class of separable, reflexive spaces has
no universal element, it is the (necessarily uncountable) increasing
union of classes which are closed under taking duals, and for which
universal separable, reflexive spaces do exist. In this paper we
answer this in the affirmative.

As in the proof in~\cite{OS:06b}, we will reduce the universality
problem to an embedding problem. We will show that any member of
$\cC_\alpha$, $\alpha\kle\omega_1$, embeds into some element of
$\cC_\beta$ which has a basis, where $\alpha\kleq\beta\kle\omega_1$
depends on $\alpha$. This result can be seen as a quantitative version
of Zippin's seminal theorem~\cite{Zippin:88} that every separable
reflexive space embeds into one with a basis. In light of our
embedding result we then only need to show that the class of elements
in $\cC_\alpha$ with a basis admits a universal separable, reflexive
space.

Our approach depends on showing that if $X$ is a space with separable
dual and if $\sz(X)\kleq \omega^{\alpha\cdot\omega}$ for some
$\alpha\kle\omega_1$, then $X$ satisfies ``subsequential upper
$T_{\alpha,c}$ estimates'', where $T_{\alpha,c}$ is the Tsirelson
space defined by the Schreier class $\cS_\alpha$ and a parameter
$c\kin (0,1)$ (the definitions of $\cS_\alpha$ and $T_{\alpha,c}$ will
be recalled in Section~\ref{section:tsirelson}). Subsequential upper
$T_{\alpha,c}$ estimates can be expressed in terms of a game played as
follows. Player~(I) starts by choosing $X_1\kin\cof(X)$, the set of
all finite-codimensional subspaces of $X$, and an integer
$k_1\kin\bn$. Player~(II) then responds by selecting $x_1\kin
S_{X_1}$, the unit sphere of $X_1$. Then~(I) chooses $X_2\kin\cof(X)$
and $k_2\kin\bn$, and~(II) chooses $x_2\kin S_{X_2}$, etc. $X$
satisfies subsequential upper $T_{\alpha,c}$ estimates if for some
$C\kle\infty$ Player~(I) has a winning strategy to force~(II) to
select $(x_i)$ satisfying
\[
\Bignorm{\sum a_ix_i}_X\leq C\Bignorm{\sum a_it_{k_i}}_{T_{\alpha,c}}
\]
for all $(a_i)\ksubset\br$, where $(t_i)$ is the unit vector basis of
$T_{\alpha,c}$. These games  are a variation of the games introduced
in~\cite{OS:02} and were defined and analysed in~\cite{OSZ2}. Using
the results therein we ultimately prove the following structure
theorem.

\begin{mainthm}
  \label{mainthm:structural}
  Let $\alpha\kle\omega_1$. For a separable, reflexive space $X$ the
  following are equivalent.
  \begin{mylist}{(ii)}
  \item
    $X\kin\cC_{\omega^{\alpha\cdot\omega}}$.
  \item
    $X$ embeds into a separable, reflexive space $Z$ with an FDD
    $(E_i)$ which satisfies subsequential
    $(\di{T}{*}{\alpha,c},T_{\alpha,c})$ estimates in $Z$ for some
    $c\kin (0,1)$.
  \end{mylist}  
\end{mainthm}

In part~(ii) ``subsequential $(\di{T}{*}{\alpha,c},T_{\alpha,c})$
estimates'' mean the following: there exists $C\kle\infty$ such that
if $(z_i)$ is a block sequence of $(E_n)$ with $\min\supp(z_i)\keq
k_i$, then
\[
C^{-1} \Bignorm{\sum \norm{z_i} t_{k_i}^*}_{\di{T}{*}{\alpha,c}} \leq
\Bignorm{\sum z_i} \leq C\Bignorm{\sum \norm{z_i}
  t_{k_i}}_{T_{\alpha,c}}\ .
\]
Of course the implication ``(ii)$\Rightarrow$(i)'' shows that the
space $Z$ lies in the same class $\cC_{\omega^{\alpha\cdot\omega}}$ as
does $X$.

Roughly we have that the Tsirelson spaces $T_{\alpha,c}$ of order
$\alpha$ form a sort of upper envelope and their duals
$\di{T}{*}{\alpha,c}$ form a lower envelope for the entire class
$\cC_{\omega^{\alpha\cdot\omega}}$. Moreover, since the spaces
$T_{\alpha,c}$ also belong to the class
$\cC_{\omega^{\alpha\cdot\omega}}$, this result is best possible.

From Theorem~\ref{mainthm:structural} and the main result
of~\cite{OSZ2} (see Theorem~\ref{thm:universal} below) we will then
deduce the following embedding and universality result.

\begin{mainthm}
  \label{mainthm:embedding}
  For each $\alpha\kle\omega_1$ every Banach space in the class
  $\cC_{\omega^{\alpha\cdot\omega}}$ embeds into a space $Z$ with a
  basis that lies in the same class $\cC_{\omega^{\alpha\cdot\omega}}$
  as does $X$.
\end{mainthm}

\begin{mainthm}
  \label{mainthm:universal}
  For each $\alpha\kle\omega_1$ there is an element of
  $\cC_{\omega^{\alpha\cdot\omega+1}}$ with a basis which is universal
  for the class $\cC_{\omega^{\alpha\cdot\omega}}$.
\end{mainthm}

Our structural result, Theorem~\ref{mainthm:structural}, will be
proved in Section~\ref{section:main} together with some more general
structural results (Theorem~\ref{thm:tight-ta},
Corollary~\ref{cor:szlenk-lower-upper-est}). We will then deduce
\emph{FDD versions} (see Theorem~\ref{thm:universal-ca2}) of our
embedding result, Theorem~\ref{mainthm:embedding}, and our universality
result, Theorem~\ref{mainthm:universal}. A result of
W.~B.~Johnson~\cite{Johnson:71} allows us to replace FDD's by bases in
our conclusions (see Theorem~\ref{thm:johnson} and the proof of
Theorems~\ref{mainthm:embedding} and~\ref{mainthm:universal}
thereafter).
In Sections~\ref{section:embedding}
to~\ref{section:szlenk} we present the relevant notation and
background material: the embedding theorem from~\cite{OSZ2}, Tsirelson
spaces, general ordinal indices and the Szlenk index. Each of these
sections begins with a brief summary of its contents. 

There is a completely different approach to universality problems that
uses tools of descriptive set theory. There have been some remarkable
achievements in Banach space theory using such techniques including
solutions of universality problems~\cite{AD,DF}. However, in order to
tackle Pe\l czy\'nski's question with this approach, one would need
the classes $\cC_\alpha$ to be analytic and this was not known. Our
results, however, now do show the following.

\begin{mainthm}
  \label{mainthm:analytic}
  For each countable ordinal $\alpha$ the class $\cC_\alpha$ is
  analytic in the Effros-Borel structure of closed subspaces of
  $C[0,1]$.
\end{mainthm}

If $\alpha$ is of the form $\omega^{\eta\cdot\omega}$ for some
$\eta\kle\omega_1$, then Theorem~\ref{mainthm:analytic} follows from
our main results and from standard facts in descriptive set theory (as
pointed out to us by C.~Rosendal). This, combined with a recent result
of P.~Dodos~\cite{Dodos} concerning analyticity of duals of analytic
classes, then gives the general case. We present this result in the
final section of our paper. We are grateful to P.~Dodos, V.~Ferenczi
and C.~Rosendal for showing us the descriptive set theoretic
implications of our results.

Let us now mention some open problems. The first one asks if
Theorem~\ref{mainthm:universal} can be sharpened.

\begin{problem}
  Is there a universal element of $\cC_{\omega^{\alpha\cdot\omega}}$
  for each $\alpha\kle\omega_1$?
\end{problem}
We do know that there is no space $Z$ in
$\cC_{\omega^{\alpha\cdot\omega}}$ such that for some $K\kge 0$ every
space in $\cC_{\omega^{\alpha\cdot\omega}}$ $K$-embeds into $Z$ (see
the remark following Theorem~\ref{thm:universal-ca}). We also know
that the answer to the above question is negative if the Tsirelson
spaces $T_{\alpha,c}$ are of bounded distortion with constant $D$
independent of $c\kin\big[\frac12,1\big)$. Of course, it is a famous,
long standing open problem whether even the Tsirelson space
$T_{1,\frac12}$ is of bounded distortion.

It is known that one only needs to consider classes $\cC_\alpha$ where
$\alpha$ is of the form $\omega^\eta$ for some $\eta\kle\omega_1$
(Theorem~\ref{thm:szlenk=weak} in Section~\ref{section:szlenk}). In
Section~\ref{section:main} we obtain embedding and universality
results for these general classes
(Theorem~\ref{thm:universal-ca}). However, these are not quite as
sharp as Theorems~\ref{mainthm:embedding} and~\ref{mainthm:universal}
above. This leads to the following questions.

\begin{problem}
  Is it true that, given $\alpha\kle\omega_1$, every space in
  $\cC_{\omega^\alpha}$ embeds into a space with a basis of the same
  class?

  Is there a universal element of $\cC_{\omega^\alpha}$ for each
  $\alpha\kle\omega_1$?
\end{problem}

\section{Embeddings into spaces with FDDs}
\label{section:embedding}

In this section we state an embedding theorem from~\cite{OSZ2}
(Theorem~\ref{thm:universal} below). This requires a fair amount of
definitions. Much of this will be used throughout the paper.

Let $Z$ be a Banach space with an FDD $E\keq(E_n)$.  For $n\kin\bn$ we
denote by $P^E_n$ the
$n$-th \emph{coordinate projection}, \ie $P^E_n\colon Z\to E_n$ is the
map defined by $\sum_i z_i\mapsto z_n$, where $z_i\kin E_i$ for all
$i\kin\bn$. For a finite set $A\ksubset\bn$ we put
$P^E_A\keq\sum_{n\in A} P^E_n$. The \emph{projection constant $K(E,Z)$
  of $(E_n)$ (in $Z$) }is defined by
\[
K=K(E,Z)=\sup_{m\leq n}\norm{P^E_{[m,n]}}\ ,
\]
where $[m,n]$ denotes the interval $\{m,m\kplus 1,\dots, n\}$ in
$\bn$. Recall that $K$ is always finite and, as in the case of bases,
we say that \emph{$(E_n)$ is bimonotone (in $Z$) }if $K\keq 1$. By
passing to the equivalent norm
\[
\tnorm{\cdot}\colon Z\to\br\ ,\qquad z\mapsto\sup_{m\leq n}
\norm{P^E_{[m,n]}(z)}\ ,
\]
we can always renorm $Z$ so that $K\keq 1$.

A sequence $(F_n)$ of finite-dimensional spaces is  called \emph{a
  blocking of $(E_n)$ }if for some sequence $m_1\kle m_2\kle\dots$ in
$\bn$ we have $F_n\keq\bigoplus_{j=m_{n-1}+1}^{m_n}E_j$ for all
$n\kin\bn$ ($m_0\keq 0$). Note that if $E\keq(E_n)$ is an FDD of a
Banach space $Z$, and if $F\keq(F_n)$ is a blocking of $(E_n)$, then
$(F_n)$ is also an FDD for $Z$ with $K(F,Z)\kleq K(E,Z)$.

For a sequence $(E_i)$ of finite-dimensional spaces
we define the vector space
\[
\coo(\oplus_{i=1}^\infty E_i)=\big\{(z_i):\,z_i\kin E_i \text{ for all
  $i\kin\bn$, and $\{i\kin\bn:\,z_i\kneq 0\}$ is finite}\big\}\ ,
\]
which is dense in each Banach space for which $(E_i)$ is an FDD. For a
set $A\ksubset\bn$ we denote  by $\bigoplus_{i\in A} E_i$ the linear
subspace of $\coo(\oplus E_i)$ generated by the elements of
$\bigcup_{i\in A}E_i$. As usual we denote the vector space of
sequences in $\br$ which are eventually zero by $\coo$. We sometimes
will consider for the same sequence $(E_i)$ of finite-dimensional
spaces different norms on $\coo(\oplus E_i)$. In order to avoid
confusion we will therefore often index the norm by the Banach space
whose norm we are using, \ie $\norm{\cdot}_Z$ denotes the norm of the
Banach space $Z$.

If $Z$ has an FDD $(E_i)$, the vector space
$\coo(\oplus_{i=1}^\infty E^*_i)$, where $E^*_i$ is the dual space of
$E_i$ for each $i\kin\bn$, can be identified in a natural way with a
$w^*$-dense subspace of $Z^*$. Note however
that the embedding $E^*_i\hookrightarrow Z^*$ is, in general, not
isometric unless $K\keq 1$. We will always consider $E^*_i$ with the
norm it inherits from $Z^*$ instead of the norm it has as the dual
space of $E_i$. We denote the norm closure of
$\coo(\oplus_{i=1}^\infty E^*_i)$ in $Z^*$ by $\Zs$. Note that $\Zs$
is $w^*$-dense in $Z^*$, the unit ball $B_{\Zs}$ norms $Z$, and
$(E_i^*)$ is an FDD of $\Zs$ having a projection constant not
exceeding $K(E,Z)$. If $K(E,Z)\keq1$, then $B_{\Zs}$ is $1$-norming
for $Z$ and $Z^{(*)(*)}\keq Z$.

For $z\kin \coo(\oplus E_i)$ we define \emph{the support $\supp_E(z)$
  of $z$ with respect to $(E_i)$ }by
\[
\supp_E(z)=\{i\kin\bn:\,P^E_i(z)\kneq 0\}\ .
\]

A sequence $(z_i)$ (finite or infinite) of non-zero
vectors in $\coo(\oplus E_i)$ is called \emph{a block sequence of
  $(E_i)$ }if
\[
\max\supp_E(z_n)<  \min\supp_E(z_{n+1})\qquad\text{whenever $n\kin\bn$
  (or $n\kle\,$length$(z_i)$)}\ .
\]
A block sequence $(z_i)$ of $(E_i)$ is called \emph{normalized (in
  $Z$) }if $\norm{z_n}_Z\keq 1$ for all $n$.

Let $\deltab\keq(\delta_i)\ksubset (0,1)$ with $\delta_i\downarrow
0$. A (finite or infinite) sequence $(z_i)$ in $S_Z$ is called
\emph{a $\deltab$-block sequence of $(E_i)$ }if there exists a
sequence $0\kleq k_0\kle k_1\kle k_2\kle\dots$ in $\bn$ such that
\[
\norm{z_n-P^E_{(k_{n-1},k_n]}(z_n)}<\delta_n\qquad\text{for all
  $n\kin\bn$ (or $n\kleq\,$length$(z_i)$)}\ .
\]

\begin{defn}
  Given two sequences $(e_i)$ and $(f_i)$ in some Banach spaces, and
  given a constant $C\kge 0$, we say that \emph{$(f_i)$ $C$-dominates
    $(e_i)$, }or that \emph{$(e_i)$ is $C$-dominated by $(f_i)$}, if
  \[
  \Bignorm{\sum a_ie_i}\leq C  \Bignorm{\sum a_if_i}\qquad\text{for
    all $(a_i)\kin\coo$}\ .
  \]  
  We say that \emph{$(e_i)$ and $(f_i)$ are $C$-equivalent }if there
  exist positive constants $A$ and $B$ with $A\kcdot B\kleq C$ such
  that $(f_i)$ $A$-dominates $(e_i)$ and is $B$-dominated by $(e_i)$.

  We say that \emph{$(f_i)$ dominates $(e_i)$, }or that \emph{$(e_i)$
    is dominated by $(f_i)$, }if there exists a constant $C\kge 0$ such
  that $(f_i)$ $C$-dominates $(e_i)$. We say that \emph{$(e_i)$ and
    $(f_i)$ are equivalent }if they are $C$-equivalent for some $C\kge
  0$.
\end{defn}

We shall now introduce certain lower and upper norm estimates for
FDD's.

\begin{defn}
  Let $Z$ be a Banach space with an FDD $(E_n)$, let $V$ be a Banach
  space with a normalized, $1$-unconditional basis $(v_i)$ and let
  $1\kleq C\kle\infty$.
  
  We say that $(E_n)$ \emph{satisfies subsequential $C$-$V$-lower
  estimates (in $Z$) }if every normalized block sequence $(z_i)$ of
  $(E_n)$ in $Z$ $C$-dominates $(v_{m_i})$, where
  $m_i\keq\min\supp_E(z_i)$ for all $i\kin\bn$, and $(E_n)$
  \emph{satisfies subsequential $C$-$V$-upper estimates (in $Z$) }if
  every normalized block sequence $(z_i)$ of $(E_n)$ in $Z$ is
  $C$-dominated by $(v_{m_i})$, where $m_i\keq\min\supp_E(z_i)$ for
  all $i\kin\bn$.

  If $U$ is another space with a normalized and 1-unconditional basis
  $(u_i)$, we say that $(E_n)$ \emph{satisfies subsequential
  $C$-$(V,U)$ estimates (in $Z$) }if it satisfies subsequential
  $C$-$V$-lower and $C$-$U$-upper estimates in~$Z$.

  We say that $(E_n)$ satisfies \emph{subsequential $V$-lower,
  $U$-upper }or \emph{$(V,U)$ estimates (in $Z$) }if for some $C\kgeq
  1$ it satisfies subsequential $C$-$V$-lower, $C$-$U$-upper or
  $C$-$(V,U)$ estimates in $Z$, respectively.
\end{defn}

We shall need a coordinate-free version of subsequential lower and
upper estimates. This can be done in terms of a game as described in
the Introduction. Another way uses infinite, countably branching
trees, and this is what we shall follow here. We define  for
$\ell\kin\bn$
\begin{align*}
  T_\ell&=\big\{(n_1,n_2,\dots,n_\ell):\,n_1\kle n_2\kle\dots\kle
  n_\ell \text{ are in }\bn\big\}
  \intertext{and}
  T_\infty&=\bigcup_{\ell=1}^\infty T_\ell\ ,\qquad
  \odd =\bigcup_{\ell=1}^\infty T_{2\ell}\ .
\end{align*}

An \emph{even tree }in a Banach space $X$ is a family
$(x_\alpha)_{\alpha\in\odd}$ in $X$. Sequences of
the form $\big(x_{(\alpha,n)}\big)_{n>n_{2\ell-1}}$, where
$\ell\kin\bn$ and $\alpha\keq (n_1, n_2,\dots, n_{2\ell-1})\kin
T_\infty$, are called \emph{nodes }of the tree. For a sequence
$n_1\kle n_2\kle\dots$ of positive integers the sequence
$\big(x_{(n_1,n_2,\dots,n_{2\ell})}\big)_{\ell=1}^\infty$ is called a
\emph{branch }of the tree.

An even tree $(x_\alpha)_{\alpha\in\odd}$ in a Banach space $X$ is
called \emph{normalized }if $\norm{x_{\alpha}}\keq 1$ for all
$\alpha\kin\odd$, and is called \emph{weakly null }if every node is a
weakly null sequence. If $X$ has an FDD $(E_n)$, then
$(x_\alpha)_{\alpha\in\odd}$ is called a \emph{block even tree of
  $(E_n)$ }if every node is a block sequence of $(E_n)$.

\begin{defn}
  \label{def:tree-est}
  Let $V$ be a Banach space with a normalized and $1$-unconditional
  basis $(v_i)$, and let $C\kin[1,\infty)$. Let $X$ be an
  infinite-dimensional Banach space. We say that \emph{$X$ satisfies
    subsequential $C$-$V$-lower tree estimates }if every normalized,
  weakly null even tree $(x_\alpha)_{\alpha\in\odd}$ in $X$ has a
  branch $\big(x_{(n_1,n_2,\dots,n_{2i})}\big)$ which $C$-dominates
  $(v_{n_{2i-1}})$.
  
  We say that $X$ \emph{satisfies subsequential $C$-$V$-upper tree
  estimates }if every normalized, weakly null even tree
  $(x_\alpha)_{\alpha\in\odd}$ in $X$ has a branch
  $\big(x_{(n_1,n_2,\dots,n_{2i})}\big)$ which is $C$-dominated by
  $(v_{n_{2i-1}})$.

  If $U$ is a second space with a $1$-unconditional and normalized
  basis $(u_i)$, we say that $X$ \emph{satisfies subsequential
  $C$-$(V,U)$ tree estimates }if it satisfies subsequential
  $C$-$V$-lower and $C$-$U$-upper tree estimates.

  We say that $X$ \emph{satisfies subsequential $V$-lower, $U$-upper
  }or \emph{$(V,U)$ tree estimates }if for some $1\kleq C\kle\infty$
  $X$ satisfies subsequential $C$-$V$-lower, $C$-$U$-upper or
  $C$-$(V,U)$ tree estimates, respectively.
\end{defn}

We next define some properties of bases which appear in the statement
of Theorem~\ref{thm:universal}.

\begin{defn}
  Let $V$ be a Banach space with a normalized, $1$-unconditional basis
  $(v_i)$ and let $1\kleq C\kle\infty$.

  We say that $(v_i)$ is \emph{$C$-block-stable }if any two normalized
  block bases $(x_i)$ and $(y_i)$ with
  \[
  \max \big( \supp (x_i)\cup\supp(y_i)\big) < \min \big( \supp
  (x_{i+1})\cup\supp(y_{i+1})\big)\qquad\text{for all }i\kin\bn
  \]
  are $C$-equivalent. We say that $(v_i)$ is \emph{block-stable }if it
  is $C$-block-stable for some constant~$C$.

  We say that $(v_i)$ is \emph{$C$-right-dominant }(respectively,
  \emph{$C$-left-dominant}) if for all sequences $m_1\kle m_2\kle
  \dots$ and $n_1\kle n_2\kle\dots$ of positive integers with
  $m_i\kleq n_i$ for all $i\kin\bn$ we have that $(v_{m_i})$ is
  $C$-dominated by (respectively, $C$-dominates) $(v_{n_i})$. We say
  that $(v_i)$ is \emph{right-dominant }or \emph{left-dominant }if for
  some $C\kgeq 1$ it is $C$-right-dominant or $C$-left-dominant,
  respectively.
\end{defn}

We are now ready to state the main embedding theorem from~\cite{OSZ2}
which we shall use in the proofs of the main results of this paper.

Let $V$ and $U$ be reflexive spaces with normalized,
$1$-unconditional, block-stable bases $(v_i)$ and $(u_i)$,
respectively, such that $(v_i)$ is left-dominant, $(u_i)$ is
right-dominant and $(v_i)$ is dominated by $(u_i)$. For each
$C\kin[1,\infty)$ let $\cA_{V,U}(C)$ denote the class of all
separable, infinite-dimensional, reflexive Banach spaces that
satisfy subsequential $C$-$(V,U)$-tree estimates. We also let
\[
\cA_{V,U}=\bigcup_{C\in[1,\infty)} \cA_{V,U}(C)\ ,
\]
which is the class of all separable, infinite-dimensional, reflexive
Banach spaces that satisfy subsequential $(V,U)$-tree estimates.

\begin{thm}[\cite{OSZ2}]
  \label{thm:universal}
  The class $\cA_{V,U}$ contains an element which is universal for the
  class.

  More precisely, for all $B,D,L,R\kin[1,\infty)$ there exists a
  constant $\Cb\keq\Cb(B,D)\kin[1,\infty)$ and for all
  $C\kin[1,\infty)$ there is a constant $K(C)\keq
  K_{B,D,L,R}(C)\kin[1,\infty)$ such that if $(v_i)$ is
  $B$-block-stable and $L$-left-dominant, if $(u_i)$ is
  $B$-block-stable and $R$-right-dominant, and if $(v_i)$ is
  $D$-dominated by $(u_i)$, then there exists $Z\kin\cA_{V,U}$ such
  that every $X\kin\cA_{V,U}(C)$ $K(C)$-embeds into $Z$, and moreover
  $Z$ has a bimonotone FDD satisfying subsequential $\Cb$-$(V,U)$
  estimates in $Z$.
\end{thm}

At some point we shall also need the following duality result.

\begin{prop}[\cite{OSZ2}]
  \label{prop:tree-est-duality}
  Assume that $U$ is a space with a normalized, $1$-unconditional
  basis $(u_i)$ which is $R$-right-dominant for some $R\kgeq 1$, and
  that $X$ is a reflexive space which satisfies subsequential
  $C$-$U$-upper tree estimates for some $C\kgeq 1$.

  Then, for any $\ve\kge 0$, $X^*$ satisfies subsequential
  $(2CR\kplus\ve)$-$\Us$-lower tree estimates.
\end{prop}

\section{Tsirelson Spaces}
\label{section:tsirelson}

When we apply Theorem~1 we shall take $U\keq T_\alpha$, the Tsirelson
space of order~$\alpha$ with parameter $\frac12$, and $V\keq
\di{T}{*}{\alpha}$, the dual of $T_\alpha$. In this section we recall
the definition and some of the properties of $T_\alpha$. At the end we
will state a combinatorial principle which will be used later on.

We begin with some preliminary definitions. We shall write
$\fin{\bn}$ for the set of all finite subsets of $\bn$ and
$\infin{\bn}$ for the set of all infinite subsets of $\bn$. These
two families will be given the product topology as subsets of
$\{0,1\}^\bn$. A family $\cF\ksubset\fin{\bn}$ is called
\emph{hereditary }if $A\kin\cF$ whenever $A\ksubset B$ and $B\kin\cF$,
and $\cF$ is called \emph{compact }if it is compact in the product
topology. Note that a hereditary family is compact if and only if it
contains no strictly ascending chains. A family
$\cF\ksubset\fin{\bn}$ is called \emph{thin }if for all $A,B\kin\cF$
we have $A\ksubset B$ implies that $A\keq B$, \ie $\cF$ contains no
two comparable (with respect to inclusion) elements.

Given $n,\ a_1\kle \dots\kle
a_n,\ b_1\kle \dots\kle b_n$ in $\bn$ we say that $\{b_1,\dots,b_n\}$
is \emph{a spread }of $\{a_1,\dots,a_n\}$ if $a_i\kleq b_i$ for $i\keq
1,\dots,n$. A family $\cF\ksubset\fin{\bn}$ is called
\emph{spreading }if every spread of every element of $\cF$ is also in
$\cF$. This is an appropriate place to make the convention that the
elements of a subset of $\bn$ will always be written in increasing
order. So, for example, when we write
$\{m_1,m_2,\dots,m_k\}\kin\fin{\bn}$, it is implicitly assumed
that $m_1\kle m_2\kle\dots\kle m_k$.

For $\cF\ksubset\fin{\bn}$ we write $\MAX (\cF)$ for the set of
maximal (with respect to inclusion) elements of $\cF$. Note that $\MAX
(\cF)$ is always a thin family.

For subsets $A$ and $B$ of $\bn$ we write $A\kle B$ if $a\kle b$ for
all $a\kin A$ and $b\kin B$. For $n\kin\bn$ and $A\ksubset \bn$ we
write $n\kle A$ if $\{n\}\kle A$. A (finite or infinite) sequence
$A_1,A_2,\dots$ of subsets of $\bn$ is called \emph{successive }if
$A_1\kle A_2\kle \dots$. Given a family $\cF\ksubset\fin{\bn}$, a
sequence $A_1,\dots,A_k$ of non-empty, finite subsets of $\bn$ is
called \emph{$\cF$-admissible }if it is successive and $\{\min
A_1,\dots,\min A_k\}\kin\cF$.

We next recall the definitions of \emph{the Schreier families
  $\cS_\alpha$ }and \emph{the fine Schreier families $\cF_\alpha$},
where $\alpha$ is a countable ordinal. We first fix for every limit
ordinal $\lambda$ a sequence $(\alpha_n)$ of ordinals with $1\kleq
\alpha_n\nearrow\lambda$. We now define the fine Schreier families
$(\cF_\alpha)_{\alpha <\omega_1}$ by recursion:
\begin{align*}
  \cF_0 &= \{ \emptyset \}\\
  \cF_{\alpha+1} &= \big\{ \{n\}\cup A:\,n\kin\bn,\ A\kin\cF_\alpha,\
  n\kle A \big\} \cup \{\emptyset\}\\
  \cF_{\lambda} &= \big\{ A\kin\fin{\bn}:\,\E n\kleq\min A,\
  A\kin\cF_{\alpha_n} \big\}\ ,
\end{align*}
where in the last line $\lambda$ is a limit ordinal and
$\alpha_n\nearrow\lambda$ is the sequence of ordinals fixed in
advance. An easy induction shows that $\cF_\alpha$ is a compact,
hereditary and spreading family for all
$\alpha\kle\omega_1$. Moreover, $(\cF_\alpha)_{\alpha<\omega_1}$ is an
``almost'' increasing chain:
\begin{equation}
  \label{eq:fine-schreier-increasing}
  \V \alpha\kleq\beta\kle\omega_1\quad \E n\kin\bn\quad \V
  F\kin\cF_\alpha\qquad\text{if $n\kleq \min F$, then
  $F\kin\cF_\beta$}\ .
\end{equation}
This can be proved by an easy induction on $\beta$. We note also that
for $A\kin\cF_\alpha\ksetminus\MAX(\cF_\alpha)$ we have
$A\cup\{n\}\kin\cF_\alpha$ for all $n\kge\max A$.

The Schreier families can now be defined by setting
$\cS_\alpha\keq \cF_{\omega^\alpha}$ for all
$\alpha\kle\omega_1$. This is not exactly how the Schreier families
are usually defined, but it gives the same families provided we are
more careful when choosing the sequences $(\alpha_n)$ with
$1\kleq \alpha_n\nearrow\lambda$ for limit ordinals
$\lambda$. At any rate, what matters is that each $\cS_\alpha$ be a
compact, hereditary and spreading family with Cantor-Bendixson index
$\omega^\alpha\kplus 1$ (see~\eqref{eq:cb-index-schreier} in
Section~\ref{section:ordinal}). Note that $\cS_1$ is the usual
Schreier family $\cS$ given by
\[
\cS = \{ A\kin\fin{\bn}:\,\abs{A}\kleq \min A \}
\]
(provided that for $\lambda\keq\omega$ we chose the sequence
$\alpha_n\keq n$).

As usual we denote by $(e_i)$ the canonical (algebraic) basis of the
vector space $\coo$ of all eventually zero scalar sequences. For
$x\keq\sum x_ie_i\kin\coo$ and for $A\ksubset \bn$ we write $Ax$ for
the obvious projection of $x$ onto $\spn \{ e_i:\,i\kin A\}$:
\[
Ax = \sum_{i\in A}x_ie_i\ .
\]
We are now ready to recall the definitions of certain Tsirelson type
spaces. For a compact, hereditary family $\cF\ksubset\fin{\bn}$ and
for $c\kin (0,1)$ there is a unique least norm on $\coo$, denoted by
$\norm{\cdot}_{\cF,c}$ such that
\[
\norm{x}_{\cF,c} = \norm{x}_\infty\vee c\kcdot\sup \Big\{ \sum
_{i=1}^n \norm{A_ix}_{\cF,c}:\,n\kin\bn,\ A_1,\dots,A_n\text{ is
  $\cF$-admissible} \Big\}
\]
for all $x\kin\coo$.
We shall write $T_{\cF,c}$ for the completion of $\coo$ in this
norm. Note that the (algebraic) basis $(e_i)$ of $\coo$ becomes a
$1$-unconditional (Schauder) basis of $T_{\cF,c}$.

For a non-zero, countable ordinal $\alpha$ and for $c\kin(0,1)$ the
space $T_{\cS_\alpha,c}$ is the \emph{Tsirelson space of order
  $\alpha$ with parameter $c$ }and we shall denote it by
$T_{\alpha,c}$. We further simplify notation in the case
$c\keq\frac12$ by letting $T_\alpha\keq T_{\alpha,\frac12}$. When
$\alpha\keq 1$ this is just (the dual of) the original Tsirelson
space~\cite{Tsirelson,FJ:74}.

We gather some properties of Tsirelson spaces in the next
proposition. In particular, we note that the unit vector bases of
$T_\alpha$ and $\di{T}{*}{\alpha}$ satisfy the conditions required in
Theorem~\ref{thm:universal}.

\begin{prop}[\cite{CJT}, \cite{LeTa:03}]
  \label{prop:tsirelson-properties}
  Let $\alpha$ be a non-zero, countable ordinal. The Tsirelson space
  $T_\alpha$ is a reflexive Banach space and $(e_i)$ is a
  $1$-unconditional, $1$-right-dominant and $B$-block-stable basis for
  $T_\alpha$, where $B$ is a constant independent of $\alpha$. 
  
  The biorthogonal functionals $(e^*_i)$ form a $1$-unconditional,
  $1$-left dominant and $B$-block-stable basis for
  $T_\alpha^*$. Moreover, $(e^*_i)$ is $D$-dominated by $(e_i)$, where
  $D$ is a universal constant.
\end{prop}

It is shown in~\cite{CJT} that Tsirelson's space $T_1$ is block-stable
(see also~\cite[Proposition~II.4]{CS}). Their argument easily carries
over to higher order Tsirelson spaces giving the same constant. A
proof is given in~\cite{LeTa:03} for an even larger class of Tsirelson type
spaces.

It is proved in~\cite[Proposition~V.10]{CS} that the unit vector basis
$(e_i)$ of Tsirelson's space $T_1$ dominates the unit vector basis of
$\ell_q$ for all $q\kge 1$. The last statement of
Proposition~\ref{prop:tsirelson-properties} now follows
immediately. (Note that $\cS_1\ksubset \cS_\alpha$, and hence the unit
vector basis of $T_1$ is $1$-dominated by the unit vector basis of
$T_\alpha$ for any $1\kleq\alpha\kle\omega_1$.)

The rest of the properties claimed in
Proposition~\ref{prop:tsirelson-properties} are immediate from the
definition of the higher order Tsirelson spaces.

We end this section by stating a combinatorial theorem of Pudl\' ak
and R\" odl which also follows from infinite Ramsey theory. This has
nothing to do with Tsirelson spaces, but as it concerns families of
finite subsets of $\bn$ this section is an appropriate place for it.

\begin{thm}[\cite{PuRo:82}]
  \label{thm:thin-ramsey}
  Let $\cF\ksubset\fin{\bn}$ be a thin family. Whenever each element
  of $\cF$  is coloured red or blue, there is an infinite subset $M$
  of $\bn$ such that $\cF\cap\fin{M}$ is monochromatic, where
  $\fin{M}$ denotes the set of all finite subsets of $M$.
\end{thm}

\section{Ordinal indices}
\label{section:ordinal}

The main aim of this section is two introduce two ordinal indices in
Banach spaces: the weak index and the block index. The former will be
related to the Szlenk index later on. In order to avoid tiresome
repetitions we begin with defining a class of ordinal indices of trees
on arbitrary sets. We then introduce the said indices as special
cases and prove a number of their properties to be used in the
sequel.

Let $X$ be an arbitrary set. We set $X^{<\omega}\keq
\bigcup_{n=0}^\infty X^n$, the set of all finite sequences in $X$,
which includes the sequence of length zero denoted by $\emptyset$. For
$x\kin X$ we shall write $x$ instead of $(x)$, \ie we identify $X$
with sequences of length~$1$ in $X$. A \emph{tree on $X$ }is a 
non-empty 
subset $\cF$ of $X^{<\omega}$ closed under taking initial segments: if
$(x_1,\dots,x_n)\kin \cF$ and $0\kleq m\kleq n$, then
$(x_1,\dots,x_m)\kin\cF$. A tree $\cF$ on $X$ is \emph{hereditary }if
every subsequence of every member of $\cF$ is also in $\cF$.

Given $\vx\keq (x_1,\dots,x_m)$ and $\vy\keq(y_1,\dots,y_n)$ in
$X^{<\omega}$, we write $(\vx,\vy)$ for the concatenation of $\vx$
and $\vy$:
\[
(\vx,\vy)=(x_1,\dots,x_m,y_1,\dots,y_n)\ .
\]
Given $\cF\ksubset X^{<\omega}$ and $\vx\kin X^{<\omega}$, we let
\[
\cF(\vx)= \{\vy\kin X^{<\omega}:\,(\vx,\vy)\kin\cF\}\ .
\]
Note that if $\cF$ is a tree on $X$, then so is $\cF(\vx)$
(unless it is empty) 
. Moreover,
if $\cF$ is hereditary, then so is $\cF(\vx)$ and $\cF(\vx)\ksubset
\cF$.

Let $X^\omega$ denote the set of all (infinite) sequences in $X$. Fix
$S\ksubset X^\omega$. For a tree $\cF$ on $X$ \emph{the $S$-derivative
  $\cF_S'$ of $\cF$ }consists of all finite sequences $\vx\kin
X^{<\omega}$ for which there is a sequence $(y_i)_{i=1}^\infty\kin S$
with $(\vx,y_i)\kin\cF$ for all $i\kin\bn$. Note that
$\cF'_S\ksubset\cF$ and that $\cF'_S$ is also a tree
(unless it is empty) 
. We then define
higher order derivatives $\cF\sder{\alpha}$ for ordinals
$\alpha\kle\omega_1$ by recursion as follows.

\begin{align*}
  \cF\sder{0} &= \cF\\
  \cF\sder{\alpha+1} &= \big( \cF\sder{\alpha}\big)'_S\qquad\text{for
    all }\alpha\kle\omega_1\\
  \cF\sder{\lambda} &= \bigcap _{\alpha <\lambda}
  \cF\sder{\alpha}\qquad \text{for a limit ordinal
  }\lambda\kle\omega_1\ .
\end{align*}

It is clear that $\cF\sder{\alpha}\ksupset \cF\sder{\beta}$ whenever
$\alpha\kleq\beta$ and that $\cF\sder{\alpha}$ is a tree
(or the empty set) 
 for all
$\alpha$. An easy induction also shows that
\[
\big(\cF(\vx)\big)\sder{\alpha} = \big(\cF\sder{\alpha}\big)
(\vx)\qquad \text{for all }\vx\kin X^{<\omega},\ \alpha\kle\omega_1\ .
\]

We now define \emph{the $S$-index $\si(\cF)$ of $\cF$ }by
\[
\si(\cF) = \min \{\alpha\kle\omega_1:\,\cF\sder{\alpha}\keq \emptyset
\}
\]
if there exists $\alpha\kle\omega_1$ with $\cF\sder{\alpha}\keq
\emptyset$, and $\si(\cF)\keq\omega_1$ otherwise.

\begin{rem}
  If $\lambda$ is a limit ordinal and
  $\cF\sder{\alpha}\kneq\emptyset$ for all $\alpha\kle\lambda$, then
  in particular $\emptyset\kin\cF\sder{\alpha}$ for all
  $\alpha\kle\lambda$, and hence
  $\cF\sder{\lambda}\kneq\emptyset$. This show that $\si(\cF)$ is
  always a successor ordinal.
\end{rem}

\begin{exs}
  1.~A hereditary family $\cF\ksubset \fin{\bn}$ can be thought of
  as a tree on $\bn$: a set $F\keq\{m_1,\dots,m_k\}\kin\fin{\bn}$ is
  identified with $(m_1,\dots,m_k)\kin\bn^{<\omega}$ (recall that
  $m_1\kle\dots\kle m_k$ by our convention of always listing the
  elements of a subset of $\bn$ in increasing order).

  Let $S$ be the set of all strictly increasing sequences in $\bn$. In
  this case the $S$-index of a compact, hereditary family
  $\cF\ksubset\fin{\bn}$ is nothing else but the Cantor-Bendixson
  index of $\cF$ as a compact topological space, which we will denote
  by $\cbi(\cF)$. We will also use the term Cantor-Bendixson
  derivative instead of $S$-derivative and use the notation
  $\cF'_{\mathrm{CB}}$ and $\cF\cbder{\alpha}$.

  \noindent
  2.~If $X$ is an arbitrary set and $S\keq X^\omega$, then the
  $S$-index of a tree $\cF$ on $X$ is what is usually called \emph{the
  order of $\cF$ }(or \emph{the height of $\cF$}) denoted by
  o$(\cF)$. Note that in this case the $S$-derivative of $\cF$
  consists of all finite sequences $\vx\kin X^{<\omega}$ for which
  there exists $y\kin X$ such that $(\vx,y)\kin\cF$.

  The function o$(\cdot)$ is the largest index: for any
  $S\ksubset X^\omega$ we have o$(\cF)\kgeq \si(\cF)$.
\end{exs}

We say that $S\ksubset X^\omega$ \emph{contains diagonals }if every
subsequence of every member of $S$ also belongs to $S$ and for
every sequence $(\vx_n)$ in $S$ with $\vx_n\keq
(x_{n,i})_{i=1}^\infty$ there exist $i_1\kle i_2\kle\dots$ in $\bn$
such that $(x_{n,i_n})_{n=1}^\infty$ belongs to~$S$.

If $S$ contains diagonals, then the $S$-index of a tree on $X$ may be
measured via the Cantor-Bendixson index of the fine Schreier families
$\big(\cF_{\alpha}\big)_{\alpha<\omega_1}$ introduced earlier. An easy
induction argument shows that
\begin{equation}
  \label{eq:cb-index-schreier}
  \cbi(\cF_\alpha) = \alpha+1\qquad \text{for all }\alpha\kle\omega_1\
  .
\end{equation}
Given a tree $\cF\ksubset\fin{\bn}$ on $\bn$, a family
$(x_F)_{F\in\cF\setminus\{\emptyset\}}$ in $X$ will always be viewed
as the tree
\[
\Big\{ \big( x_{\{m_1\}}, x_{\{m_1,m_2\}}, \dots,x_{\{m_1,
  m_2,\dots,m_k\}} \big):\,k\kgeq 0,\ \{m_1,\dots,m_k\}\kin\cF \Big\}\
.
\]
on $X$. 

\begin{prop}
  \label{prop:S-index-fine-schreier}
  Let $X$ be an arbitrary set and let $S\ksubset X^\omega$. If $S$
  contains diagonals, then for a tree $\cF$ on $X$ and for a countable
  ordinal $\alpha$ the following are equivalent.
  \begin{mylist}{(ii)}
  \item
    $\alpha<\si(\cF)$.
  \item
    There is a family $\big(x_F\big)
    _{F\in\cF_\alpha\setminus\{\emptyset\}} \ksubset \cF$ such that
    for all $F\kin \cF_\alpha\ksetminus \MAX(\cF_\alpha)$ the sequence
    $\big(x_{F\cup\{n\}}\big) _{n>\max F}$ is in~$S$.
  \end{mylist}
\end{prop}

\begin{proof}
  ``(ii)$\Rightarrow$(i)'' An easy induction on $\beta\kle\omega_1$
  shows that for all $F\keq\{m_1,\dots,m_k\}
  \kin(\cF_\alpha)\cbder{\beta}$ we have
  \[
  \big( x_{\{m_1\}}, x_{\{m_1,m_2\}}, \dots,x_{\{m_1,
    m_2,\dots,m_k\}} \big) \in \cF\sder{\beta}\ .
  \]
  It follows that $\si(\cF) \kgeq \cbi(\cF_{\alpha})\kge\alpha$.

  \noindent
  ``(i)$\Rightarrow$(ii)'' We prove this by induction on
  $\alpha$. When $\alpha\keq 0$, statement~(ii) says that
  $\emptyset\kin\cF$, which does follow from $0\kle\si(\cF)$.

  Next assume that $\alpha\kplus 1\kle\si(\cF)$. Then
  $\cF\sder{\alpha+1}\kneq\emptyset$, so in particular we have
  $\emptyset\kin\cF\sder{\alpha+1}$. It follows that there is a
  sequence $(x_i)_{i=1}^\infty\kin S$ such that
  $x_i\kin\cF\sder{\alpha}$ for all $i\kin\bn$. Hence
  $\big(\cF\sder{\alpha}\big)(x_i)\keq
  \big(\cF(x_i)\big)\sder{\alpha}$ is non-empty, and
  $\si\big(\cF(x_i)\big)\kge \alpha$. By the induction hypothesis, for
  each $i\kin\bn$ there is a family $\big(y_{i,F}\big)
  _{F\in\cF_\alpha\setminus\{\emptyset\}} \ksubset \cF(x_i)$ such that
  for all $F\kin \cF_\alpha\ksetminus \MAX(\cF_\alpha)$ the sequence
  $\big(y_{i,F\cup\{n\}}\big) _{n>\max F}$ is in~$S$.

  Now for each $F\keq\{m_1,\dots,m_k\}\kin\cF_{\alpha+1}$ define 
  \[
  x_F = \begin{cases}
    x_i & \text{if }k\keq 1\text{ and }m_1\keq i\ ,\\
    y_{i,\{m_2,m_3,\dots, m_k\}} & \text{if }k\kge 1\text{ and
    }m_1\keq i\ .\\
  \end{cases}
  \]
  It is routine to verify that statement~(ii) with $\alpha\kplus 1$
  replacing $\alpha$ holds.

  Finally, let $\lambda$ be a limit ordinal, and assume that
  $\lambda\kle\si(\cF)$. Let $(\alpha_n)$ be the sequence of ordinals
  with $1\kleq \alpha_n\nearrow \lambda$ chosen in the definition of
  the fine Schreier family $\cF_\lambda$. By the induction hypothesis,
  for each $n\kin\bn$ there is a family $\big(y_{n,F}\big)
  _{F\in\cF_{\alpha_n}\setminus\{\emptyset\}} \ksubset\cF$ such that
  for all $F\kin \cF_{\alpha_n}\ksetminus \MAX(\cF_{\alpha_n})$ the
  sequence $\big(y_{n,F\cup\{i\}}\big) _{i>\max F}$ is in~$S$.
  In particular we have $\big(y_{n,\{i\}}\big)_{i=1}^\infty\kin S$ for
  all $n\kin\bn$. Since $S$ contains diagonals there exist $i_1\kle
  i_2\kle\dots$ in $\bn$ such that
  $\big(y_{n,\{i_n\}}\big)_{n=1}^\infty\kin S$. We
  can also ensure that if $m\kleq n,\ F\kin\cF_{\alpha_m}$ and
  $i_n\kleq\min F$, then $F\kin\cF_{\alpha_n}$
  (see~\eqref{eq:fine-schreier-increasing}).

  Now for each $F\keq\{m_1,\dots,m_k\}\kin\cF_{\lambda}$ define
  \[
  x_F = \begin{cases}
    y_{n,\{i_n\}} & \text{if }k\keq 1\text{ and }m_1\keq n\ ,\\
    y_{n,\{i_n, i_n+m_2,i_n+m_3,\dots, i_n+m_k\}} & \text{if }k\kge
    1\text{ and }m_1\keq n\ .\\
  \end{cases}
  \]
  It is again routine to verify that statement~(ii) with $\lambda$ in
  the place of $\alpha$ follows.
\end{proof}

The set $S\keq X^\omega$ for an arbitrary set $X$, and the set $S$
used to define the Cantor-Bendixson index for a compact, hereditary
family in $\fin{\bn}$ trivially contain diagonals. This will also be
(mostly) the case in the following two examples of $S$-indices in
Banach spaces.

\begin{exs}
  1.~\emph{The weak index. }Let $X$ be a separable Banach space. Let
  $S$ be the set of all weakly null sequences in $S_X$, the unit
  sphere of $X$. We call the $S$-index of a tree $\cF$ on $S_X$
  \emph{the weak index of $\cF$ }and we shall denote it by
  $\wi(\cF)$. We shall use the term \emph{weak derivative }instead of
  $S$-derivative and use the notation $\cF'_{\mathrm{w}}$ and
  $\cF\wder{\alpha}$.

  When the dual space $X^*$ is separable, the weak topology on
  the unit ball $B_X$ of $X$ is metrizable. Hence in this case the set
  $S$ contains diagonals and
  Proposition~\ref{prop:S-index-fine-schreier} applies.

  \noindent
  2.~\emph{The block index. }Let $X$ be a Banach space with an FDD
  $E\keq(E_i)$. A \emph{block tree of $(E_i)$ in $Z$ }is a tree $\cF$
  on $S_X$ such that every element of $\cF$ is a (finite) block
  sequence of $(E_i)$. Let $S$ be the set of all normalized, infinite
  block sequences of $(E_i)$ in $Z$. We call the $S$-index of a block
  tree $\cF$ of $(E_i)$ \emph{the block index of $\cF$ }and we shall
  denote it by $\bli(\cF)$. We shall use the term \emph{block
  derivative }instead of $S$-derivative and use the notation
  $\cF'_{\mathrm{bl}}$ and $\cF\blder{\alpha}$. Note that the set $S$
  contains diagonals, and hence
  Proposition~\ref{prop:S-index-fine-schreier} applies.

  Note that $(E_i)$ is a shrinking FDD of $X$ if and only if every
  element of $S$ is weakly null. In this case we have
  \begin{equation}
    \label{eq:bli-wi}
      \bli(\cF) \leq \wi(\cF)
  \end{equation}
  for any block tree $\cF$ of $(E_i)$ in $Z$. The converse is false in
  general, but it is true up to perturbations and without the
  assumption that $(E_i)$ is shrinking (see the remark preceding
  Proposition~\ref{prop:compression} below).
\end{exs}

\begin{rem}
  If $(E_i)$ and $(F_i)$ are two different FDDs of the Banach space
  $X$, then the corresponding block indices they give rise to may well
  be different in general. However, it is clear that if $(F_i)$ is a
  blocking of $(E_i)$, then they do yield the same block index. Since
  this is exactly the kind of situation in which we shall use the
  block index in this paper, we did not incorporate the underlying FDD
  in the notation for block derivatives and for the block index.
\end{rem}

In the next section we will relate the Szlenk index to the weak index
of certain trees. In the rest of this section we prove two
results. The first one is a perturbation result: it concerns the weak
index of the `fattening' of a tree. The second result relates the
block index of block trees in Banach spaces to the Cantor-Bendixson
index of compact, hereditary families in $\fin{\bn}$. It is a kind
of discretization result.

Let $X$ be a separable Banach space. For a tree $\cF\ksubset
S_X^{<\omega}$ and for $\veb\keq(\ve_i)\ksubset(0,1)$ we write
\[
\cF^X_{\veb} = \big\{ (x_i)_{i=1}^n\kin S_X^{<\omega}:\,n\kin\bn,\ \E
(y_i)_{i=1}^n\kin \cF,\ \norm{x_i\kminus y_i}\kleq \ve_i\text{ for
}i\keq 1,\dots,n \big\}\ .
\]

\begin{prop}
  \label{prop:fat-index}
  Let $X\ksubset Y$ be Banach spaces with separable duals, and let
  $\cF\ksubset S_X^{<\omega}$ be a tree on $S_X$. Then for all
  $\veb\keq(\ve_i)\ksubset (0,1)$ we have $\wi(\cF^Y_{\veb})\kleq
  \wi(\cF^X_{5\veb})$.
\end{prop}

\begin{proof}
  By a theorem of Zippin~\cite{Zippin:88}, $Y$ embeds into a Banach space with
  a shrinking FDD. So without loss of generality we may assume that
  $Y$ itself has a shrinking FDD $E\keq (E_i)$. Let $K\keq K(E,Y)$,
  the projection constant of $E$ in $Y$, and set
  \[
  X_m = X\cap \overline{\textstyle\bigoplus_{j=m+1}^\infty E_j}\qquad
  (m\kin\bn)\ .
  \]
  For $\ve\kin (0,1)$ and for $A\ksubset S_Y$ we define $A^Y_\ve$ by
  \[
  A^Y_\ve = \{ y\kin S_Y:\,\E x\kin A\text{ with } \norm{x\kminus
    y}\kleq\ve\}\ .
  \]
  We shall need the following lemma.
  \begin{lem}
    \label{lem:w-null-perturbation}
    Let $\ve\kin (0,1)$ and let $(y_i)$ be a weakly null sequence in
    $\big(S_X\big)^Y_\ve$. Then there is a weakly null sequence
    $(x_i)$ in $S_X$ and a subsequence $(y'_i)$ of $(y_i)$ such that
    $\norm{x_i\kminus y'_i}\kleq 4\ve$ for all $i\kin\bn$.
  \end{lem}
  \begin{proof}
    Fix $\eta\kge 0$ such that
    \[
    \ve'=(1+\eta)(2\eta+\ve)<1\qquad\text{and}\qquad 2\ve'+\ve<4\ve\ .
    \]
    Let $m\kin\bn$. Since $\big(X/X_m\big)^* \cong X_m^\perp$ is a
    finite-dimensional subspace of $X^*$, there is a finite subset
    $A_m$ of $B_{X^*}$ such that
    \[
    d(x,X_m)\leq (1\kplus\eta) \cdot \max_{f\in A_m} f(x)\qquad
    \text{for all }x\kin X\ .
    \]
    Let $B_m$ be a finite subset of $B_{Y^*}$ containing a Hahn-Banach
    extension to $Y$ of each element of $A_m$. Then choose
    $n(m,\eta)\kin\bn$ such that
    \[
    \norm{g - P^{E^*}_{[1,n(m,\eta)]}(g)} < \eta\qquad \text{for all
    }g\kin B_m\ .
    \]
    Now let $(y'_i)$ be a subsequence of $(y_i)$ such that\
    \[
    \norm{P^E_{[1,n(m,\eta)]}(y'_m)} <\eta\qquad\text{for all
    }m\kin\bn\ .
    \]
    For each  $m\kin\bn$ choose $z_m\kin S_X$ with $\norm{z_m\kminus
      y'_m}\kleq\ve$. We have
    \begin{align*}
      d(z_m,X_m) & \leq (1\kplus\eta)\cdot \max_{g\in B_m} g(z_m)\\
      & \leq (1\kplus\eta)\cdot \Big( \max_{g\in B_m} g(y'_m) + \ve
      \Big)\\
      & < (1\kplus\eta)\cdot \Big( \max_{g\in B_m}
      P^{E^*}_{[1,n(m,\eta)]}(g)(y'_m) + \eta +\ve \Big)\\
      & \leq (1\kplus\eta)\cdot \Big(
      \norm{P^E_{[1,n(m,\eta)]}(y'_m)}+\eta +\ve \Big)
      \leq (1\kplus\eta)\cdot ( 2\eta +\ve )=\ve'\ .
    \end{align*}
    Choose $\xt_m\kin X_m$ such that $\norm{\xt_m\kminus z_m}\kle
    \ve'$, and set $x_m\keq \frac{\xt_m}{\norm{\xt_m}}$. An easy
    computation shows that
    \[
    \norm{x_m - y'_m} < 2\ve'+\ve < 4\ve\qquad\text{for all }m\kin\bn\
    .
    \]
    Since $(E_i)$ is shrinking, it follows that the sequence $(x_i)$
    is weakly null.
  \end{proof}
  We now continue with the proof of Proposition~\ref{prop:fat-index}.
  Let $\veb\keq(\ve_i)\ksubset (0,1)$. It is
  enough to show that if $\alpha\kle\wi(\cF^Y_{\veb})$, then
  $\alpha\kle \wi(\cF^X_{5\veb})$. Now if $\alpha\kle\wi(\cF^Y_{\veb})$
  holds, then by Proposition~\ref{prop:S-index-fine-schreier} there
  is a family $\big(y_F\big) _{F\in\cF_\alpha\setminus\{\emptyset\}}
  \ksubset\cF^Y_{\veb}$ such that for all $F\kin \cF_\alpha\ksetminus
  \MAX(\cF_\alpha)$ the sequence $\big(y_{F\cup\{n\}}\big) _{n>\max
  F}$ is weakly null.

  Given a spreading family $\cF\ksubset\fin{\bn}$ we will call a
  function $F\mapsto F'\colon\cF\to\cF$ \emph{a pruning function }if
  for every $F\keq\{m_1,\dots,m_\ell\}\kin\cF$ we have
  $F'\keq\{m_1',\dots,m_\ell'\}$ is a spread of $F$ and $\{m_1,\dots,
  m_k\}'\keq \{m_1',\dots,m_k'\}$ for each $k\keq 1,\dots,\ell$. 
  Now by repeated applications of Lemma~\ref{lem:w-null-perturbation}
  we can find a family $\big(x_F\big)
  _{F\in\cF_\alpha\setminus\{\emptyset\}} \ksubset S_X$ and a pruning
  function $F\mapsto F'\colon\cF_\alpha\to\cF_\alpha$ such that
  \[
  \big(x_{F\cup\{n\}}\big)_{n>\max F}\quad\text{is weakly null for all
  }F\kin \cF_\alpha\ksetminus\MAX(\cF_\alpha) \,
  \]
  and
  \[
  \norm{x_F - y_{F'}} \leq 4\ve_i\quad \text{for all }i\kin\bn\text{
  and for all }F\kin\cF_{\alpha}\text{ with }\abs{F}\keq i\ .
  \] 
  The last line implies that $\big(x_F\big)
  _{F\in\cF_\alpha\setminus\{\emptyset\}} \ksubset\cF^X_{5\veb}$, and
  hence by Proposition~\ref{prop:S-index-fine-schreier} we have
  $\alpha\kle\wi(\cF^X_{5\veb})$, as required.
\end{proof}

Let $Z$ be a Banach space with an FDD $E\keq (E_i)$, and let $\cF$ be
a block tree of $(E_i)$ in $Z$. Let us write $\Sigma(E,Z)$ for the set
of all finite, normalized block sequences on $(E_i)$ in $Z$. For
$\veb\keq(\ve_i)\ksubset (0,1)$ we let
\[
\cF^{E,Z}_{\veb} = \cF^Z_{\veb} \cap \Sigma(E,Z)\ ,
\]
\ie $\cF^{E,Z}_{\veb}$ is the restriction to block sequences of the
$\veb$-`fattening' of $\cF$ in $Z$:
\[
\cF^{E,Z}_{\veb} = \big\{ (x_i)_{i=1}^n\kin \Sigma(E,Z):\,n\kin\bn,\ \E
(y_i)_{i=1}^n\kin\cF,\ \norm{x_i\kminus y_i}\kleq\ve_i\text{ for
}i\keq 1,\dots,n \big\}\ .
\]
We also define \emph{the compression $\cFt$ of $\cF$ }by
\[
\cFt=\big\{ F\kin\fin{\bn}:\,\E (z_i)_{i=1}^{\abs{F}}\kin\cF,\
F\keq\{\min\supp_E(z_i):\,i\keq 1,\dots,\abs{F} \} \big\}\ .
\]

\begin{rem}
  Having introduced the above notation, we can now write down a sort
  of converse for~\eqref{eq:bli-wi}. If $Z$ is a Banach space with an
  FDD $E\keq (E_i)$ and $\cF$ is a block tree of $(E_i)$ in $Z$, then
  \[
  \wi(\cF) \leq \bli(\cF^{E,Z}_{\veb}) 
  \]
  for all $\veb\keq(\ve_i)\ksubset(0,1)$. Indeed, if
  $\alpha\kle\wi(\cF)$, then by
  Proposition~\ref{prop:S-index-fine-schreier} there is a family
  $\big(x_F\big) _{F\in\cF_\alpha\setminus\{\emptyset\}} \ksubset \cF$
  such that for all $F\kin \cF_\alpha\ksetminus \MAX(\cF_\alpha)$ the
  sequence $\big(x_{F\cup\{n\}}\big) _{n>\max F}$ is weakly null. By
  standard perturbation arguments we get a pruning function $F\mapsto
  F'\colon\cF_\alpha\to\cF_\alpha$ and a family $\big(y_F\big)
  _{F\in\cF_\alpha\setminus\{\emptyset\}}$ in $S_Z$ such that for all
  $F\kin \cF_\alpha\ksetminus \MAX(\cF_\alpha)$ the sequence
  $\big(y_{F\cup\{n\}}\big) _{n>\max F}$ is a block sequence, and for
  all $F\kin\cF_\alpha$ we have $\norm{x_{F'}\kminus y_F}
  \kle\ve_{\abs{F}}$. It follows by
  Proposition~\ref{prop:S-index-fine-schreier} that
  $\alpha\kle\bli(\cF^{E,Z}_{\veb})$.
\end{rem}

\begin{prop}
  \label{prop:compression}
  Let $Z$ be a Banach space with an FDD $E\keq(E_i)$. Let $\cF$ be a
  hereditary block tree of $(E_i)$ in $Z$. Then for all $\veb\keq
  (\ve_i)\ksubset(0,1)$ and for all limit ordinals $\alpha$, if
  $\bli\big( \cF^{E,Z}_{\veb} \big)\kle\alpha$, then $\cbi\big( \cFt
  \big)\kle\alpha$.
\end{prop}
The proof consists of two parts. We first replace block sequences of
$(E_i)$ with sequences of finite subsets of $\bn$
(Lemma~\ref{lem:supp-block-tree}), and then prove a discrete
compression result (Lemma~\ref{lem:discrete-compression}). Before we
begin we need to extend the notion of block index and related notions
to a discrete setting. We write $\Sigma$ for the set of all finite
successive sequences in $\fin{\bn}\ksetminus\{\emptyset\}$ and $S$
for the set of all infinite successive sequences in
$\fin{\bn}\ksetminus\{\emptyset\}$. A tree $\cG\ksubset \Sigma$ on
$\fin{\bn}$ will be called a block tree in $\fin{\bn}$, and its
$S$-index will be called the block index of $\cG$ denoted by
$\bli(\cG)$. We shall also use the term block derivative and the
notation $\cG'_{\mathrm{bl}},\ \cG\blder{\alpha}$ just like in the
Banach space case.

\begin{lem}
  \label{lem:supp-block-tree}
  Let $Z$ be a Banach space with an FDD $E\keq(E_i)$. Let $\cF$ be a
  block tree of $(E_i)$ in $Z$. Let
  \[
  \supp \cF = \big\{ (A_i)_{i=1}^n\kin\Sigma:\,n\kin\bn,\ \E
  (z_i)_{i=1}^n\kin\cF,\ \supp_E (z_i)\keq A_i\text{ for
  }i\keq 1,\dots,n\big\}\ .
  \]
  Then for all $\veb\keq(\ve_i)\ksubset(0,1)$ we have
  \[
  \bli(\supp \cF)\kleq
  \bli\big(\cF^{E,Z}_{\veb}\big)\ .
  \]
\end{lem}

\begin{proof}
  To simplify notation we are going to write $\cF_{\veb}$ instead of
  $\cF^{E,Z}_{\veb}$. We show by induction that for all
  $\alpha\kle\omega_1$ we have
  \begin{equation} 
    \label{eq:supp-block-tree;induction}
    \big( \supp \cF \big)\blder{\alpha} \subset \supp \Big( \big(
    \cF_{\veb} \big) \blder{\alpha} \Big)
    \qquad\V \cF\ ,\ \V \veb\ .
  \end{equation}
  Lemma~\ref{lem:supp-block-tree} will follow immediately. We begin
  with the case $\alpha\keq 1$. Let
  $(A_1,\dots,A_n)\kin \big(\supp\cF\big)'_{\mathrm{bl}}$. Then there
  is an infinite successive sequence $(B_i)$ in $\fin{\bn}\ksetminus
  \{\emptyset\}$ such that $(A_1,\dots, A_n,B_i)\kin\supp\cF$ for all
  $i\kin\bn$. Choose $(z_{1,i},\dots,z_{n,i},z_i)\kin \cF$ such that
  $\supp_E (z_{k,i})\keq  A_k$ for $k\keq 1,\dots,n$ and $\supp_E
  (z_i)\keq B_i$. By compactness, for some $i_0\kin\bn$ we have
  $(z_{1,i_0},\dots,z_{n,i_0},z_i)\kin \cF_{\veb}$ for infinitely many
  $i\kin\bn$. It follows that $(z_{1,i_0},\dots,z_{n,i_0})\kin
  \big(\cF_{\veb}\big)'_{\mathrm{bl}}$ and $(A_1,\dots,A_n)\kin \supp
  \Big( \big( \cF_{\veb} \big)'_{\mathrm{bl}} \Big)$, as required.

  In the inductive step we shall use the fact that $\big(
  \cF\blder{\alpha} \big)_{\veb} \ksubset \big( \cF_{\veb} \big)
  \blder{\alpha}$ for all $\cF,\ \veb,\ \alpha$, which can be verified
  by an easy induction. Assume now
  that~\eqref{eq:supp-block-tree;induction} holds. We then have
  \begin{align*}
    \big( \supp \cF \big)\blder{\alpha+1} &= \Big( \big( \supp \cF
    \big)\blder{\alpha} \Big) '_{\mathrm{bl}} \subset \bigg( \supp
    \Big( \big( \cF_{\veb} \big) \blder{\alpha} \Big)
    \bigg) '_{\mathrm{bl}}\\
    \intertext{(by the induction hypothesis)}
    &\subset \supp \Big( \big(\cH_{\veb}\big)
    '_{\mathrm{bl}} \Big)\\
    \intertext{(by the case $\alpha\keq 1$, where $\cH\keq \big(
      \cF_{\veb} \big) \blder{\alpha}$)}
    &\subset \supp \Big( \big( \cF_{2\veb}
    \big)\blder{\alpha+1} \Big)\ .
  \end{align*}
  This proves~\eqref{eq:supp-block-tree;induction} with $\alpha$
  replaced by $\alpha\kplus 1$.

  Finally, let $\lambda$ be a limit ordinal and assume
  that~\eqref{eq:supp-block-tree;induction} holds for all
  $\alpha\kle\lambda$. We have
  \[
  \big( \supp\cF \big)\blder{\lambda} = \bigcap _{\alpha <\lambda}
  \big(\supp \cF\big)\blder{\alpha} \subset \bigcap _{\alpha <\lambda}
  \supp \Big( \big(\cF_{\veb} \big) \blder{\alpha}\Big)
  \subset \supp \Big( \big( \cF_{2\veb}
  \big)\blder{\lambda} \Big)\ .
  \]
  The first inclusion follows from the induction hypothesis. To see
  the second inclusion fix a sequence $(\alpha_i)$ of ordinals with
  $\alpha_i\nearrow\lambda$, and assume that
  \[
  (A_1,\dots,A_n)\in \supp \Big( \big( \cF_{\veb}
  \big)\blder{\alpha_i} \Big) \qquad\text{for all }i\kin\bn\ .
  \]
  For each $i\kin\bn$ choose
  \[
  (z_{1,i},\dots,z_{n,i}) \in \big( \cF_{\veb}\big)
  \blder{\alpha_i}
  \]
  such that $\supp_E (z_{k,i})\keq A_k$ for $k\keq 1,\dots,n$. By
  compactness, we find $i_0\kin\bn$ such that for infinitely many
  $i\kin\bn$ we have
  \[
  \norm{z_{k,i}-z_{k,i_0}} < \ve_k\qquad\text{for }k\keq 1,\dots,n\
  .
  \]
  It follows that $(z_{1,i_0},\dots,z_{n,i_0}) \in \big(
  \cF_{2\veb}\big) \blder{\alpha_i}$ for infinitely many $i\kin\bn$,
  and hence $(z_{1,i_0},\dots,z_{n,i_0}) \in \big(
  \cF_{2\veb}\big) \blder{\lambda}$. In turn this implies that
  $(A_1,\dots,A_n)\kin \supp \big( \big( \cF_{2\veb}
  \big)\blder{\lambda} \big)$, as required.
\end{proof}

\begin{lem}
  \label{lem:discrete-compression}
  Let $\cG\ksubset\Sigma$ be a hereditary block tree in $\fin{\bn}$,
  and let
  \[
  \min\cG = \big\{ F\kin\fin{\bn}:\,\E
  (A_1,\dots,A_{\abs{F}})\kin\cG,\ F\keq\{ \min A_i:\,i\keq
  1,\dots,\abs{F}\} \big\}\ .
  \]
  Then for any limit ordinal $\alpha$, if $\bli(\cG)\kle\alpha$, then
  $\cbi(\min\cG)\kle\alpha$.
\end{lem}

\begin{proof}
  We are going to show the following three statements.
  \begin{mylist}{(iii)}
  \item
    For all $n\kle\omega$ we have $\big(\min\cG\big)\cbder{2n+2}
    \ksubset \min \big( \cG\blder{n+1} \big)$.
  \item
    Let $\alpha$ be a limit ordinal. If
    \[
    \big( \min\cG\big)\cbder{\alpha} \subset \bigcap _{\beta<\alpha}
    \min \big( \cG\blder{\beta}\big)\ ,
    \]
    then for all $n\kle\omega$ we have
    \[
    \big( \min\cG \big) \cbder{\alpha+2n+1} \subset \min \big(
    \cG\blder{\alpha+n} \big)\ .
    \]
  \item
    For every limit ordinal $\alpha\kle\omega_1$ we have
    \begin{equation}
      \label{eq:discrete-block-tree-compression;main}
      \big( \min\cG\big)\cbder{\alpha} \subset \bigcap _{\beta<\alpha}
      \min \big( \cG\blder{\beta}\big)\ .
    \end{equation}
  \end{mylist}
  Since the functions $\cbi(\cdot)$ and $\bli(\cdot)$ only take
  successor ordinals, statement (iii) implies the lemma
  immediately. We start by presenting a proof of~(i) and~(ii) in the
  case $n\keq 0$. The general case in both parts follows by an easy
  induction.

  Let $(m_1,\dots,m_k)$ be an element of $\big( \min\cG
  \big)''_{\mathrm{CB}}$. Then there exist $m\kin\bn$ and an infinite
  subset $N$ of $\bn$ with $m_k\kle m\kle \min N$ such that
  \[
  (m_1,\dots,m_k,m,n)\in \min\cG \qquad \text{for all }n\kin N\ .
  \]
  For each $n\kin N$ choose $\big( A_{1,n},\dots,A_{k,n},B_n,C_n \big)
  \in \cG$ such that
  \[
  \big( \min A_{1,n},\dots,\min A_{k,n},\min B_n,\min C_n \big) =
  (m_1,\dots,m_k,m,n)\ .
  \]
  After passing to a subsequence we can assume that $A_{j,n}\keq A_j$
  for $j\keq 1,\dots,k$ and for all $n\kin N$. Since $\cG$ is
  hereditary, we have $(A_1,\dots,A_k,C_n)\kin\cG$ for all $n\kin N$,
  and hence $(A_1,\dots,A_k)\kin\cG'_{\mathrm{bl}}$. It follows that
  $(m_1,\dots,m_k)\kin\min\big( \cG'_{\mathrm{bl}} \big)$, which
  completes the proof of~(i), $n\keq 0$.

  To show (ii), $n\keq 0$, fix a sequence $(\beta_i)$ of ordinals with
  $\beta_i\nearrow \alpha$. Pick an element $(m_1,\dots,m_k)\kin \big(
  \min\cG \big) \cbder{\alpha+1}$. Then there exist $m\kin\bn$ with
  $m_k\kle m$ such that
  \[
  (m_1,\dots,m_k,m)\in \big( \min\cG \big) \cbder{\alpha} \subset
  \min \big( \cG\blder{\beta_i} \big)\qquad
  \text{for all }i\kin \bn\ .
  \]
  For each $i\kin \bn$ choose $\big( A_{1,i},\dots,A_{k,i},B_i\big)
  \in \cG\blder{\beta_i}$ such that
  \[
  \big( \min A_{1,i},\dots,\min A_{k,i},\min B_i\big) =
  (m_1,\dots,m_k,m)\ .
  \]
  After passing to a subsequence we can assume that $A_{j,i}\keq A_j$
  for $j\keq 1,\dots,k$ and for all $i\kin\bn$. Since
  $(A_1,\dots,A_k)\kin \cG\blder{\beta_i}$ for all $i\kin\bn$, we have
  $(A_1,\dots,A_k)\kin\cG\blder{\alpha}$ and $(m_1,\dots,m_k)\kin\min
  \big( \cG\blder{\alpha} \big)$, as required.

  Finally, we are going to show (iii) by induction on
  $\alpha$. It follows from~(i)
  that~\eqref{eq:discrete-block-tree-compression;main} holds for
  $\alpha\keq \omega$. Moreover,
  if~\eqref{eq:discrete-block-tree-compression;main} holds for a
  limit ordinal $\alpha$, then it also holds for
  $\alpha\kplus\omega$ by~(ii). Finally, assume that $\alpha$ is the
  limit of a strictly increasing sequence $(\alpha_n)$ of non-zero
  limit ordinals and that~(iii) holds with $\alpha$ replaced by
  $\alpha_n$ for all $n\kin\bn$. Then in particular, we have
  \[
  \big( \min\cG \big)\cbder{\alpha_{n+1}} \subset \min \big(
  \cG\blder{\alpha_n} \big)\qquad\text{for all }n\kin\bn\ ,
  \]
  from which~(iii) follows immediately for $\alpha$.
\end{proof}

\section{The Szlenk index}
\label{section:szlenk}

Here we recall the definition and basic properties of the Szlenk
index. We then recall or prove further properties that are relevant
for our purposes.
 
Let $X$ be a separable  Banach space, and let $K$ be
a non-empty, $w^*$-compact subset of $X^*$. For $\ve\kge 0$ set
\[
K'_\ve = \{ x^*\kin K:\,\V \text{$w^*$-neighbourhoods $U$ of }
x^*\text{ we have diam}(K\cap U)\kge\ve\}\ ,
\]
where diam$(K\cap U)$ denotes the norm-diameter of $K\cap U$. We now
define $K^{(\alpha)}_\ve$ for each countable ordinal $\alpha$ by
recursion as follows:
\begin{align*}
  K^{(0)}_\ve &= K\\
  K^{(\alpha+1)}_\ve &= \big( K^{(\alpha)}_\ve
  \big)'_\ve\qquad\text{for all }\alpha\kle\omega_1\\
  K^{(\lambda)}_\ve &= \bigcap _{\alpha <\lambda}
  K^{(\alpha)}_\ve\qquad \text{for a limit ordinal
  }\lambda\kle\omega_1\ .
\end{align*}
Next, we associate to $K$ the following ordinal indices:
\[
\eta(K,\ve) = \sup \{
\alpha\kle\omega_1:\,K^{(\alpha)}_\ve\kneq\emptyset \}\ ,
\]
and
\[
\eta(K) = \sup _{\ve>0} \eta(K,\ve)\ .
\]
Finally, we define \emph{the Szlenk index $\sz(X)$ of }$X$ to be
$\eta(B_{X^*})$, where $B_{X^*}$ is the unit ball of $X^*$.
\begin{rem}
  The original definition in~\cite{Szlenk:68} is slightly different, but it
  gives the same ordinal index.
\end{rem}

Szlenk used his index to show that there is no separable, reflexive
space universal for the class of all separable, reflexive spaces. This
result follows immediately from the following properties of the
function $\sz(\cdot)$.
\begin{thm}[\cite{Szlenk:68}]
  \label{thm:szlenk-properties}
  Let $X$ and $Y$ be separable Banach spaces.
  \begin{mylist}{(iii)}
  \item
    $X^*$ is separable if and only if $\sz(X)\kle\omega_1$.
  \item
    If $X$ isomorphically embeds into $Y$, then $\sz(X)\kleq \sz(Y)$.
  \item
    For all $\alpha\kle\omega_1$ there exists a separable, reflexive
    space with Szlenk index at least $\alpha$.
  \end{mylist}
\end{thm}

We next restate in one theorem a number of results from~\cite{AJO} in
our terminology. It includes an expression of the Szlenk index in
terms of the weak index of certain trees.

\begin{thm}[\cite{AJO}]
  \label{thm:szlenk=weak}
  Let $X$ be a separable, infinite-dimensional Banach space not
  containing $\ell_1$. For $\rho\kin (0,1)$ let
  \[
  \cF_\rho = \bigg\{ (x_i)_{i=1}^n\kin S_X^{<\omega}:\,n\kin\bn,\
  \Bignorm{\sum_{i=1}^n a_ix_i}\kgeq\rho \sum_{i=1}^n a_i\ \V
  (a_i)_{i=1}^n\ksubset \br^+\bigg\}\ .
  \]
  Then
  \[
  \sz(X)\keq \sup_{\rho>0} \wi(\cF_\rho)\ .
  \]
  Moreover, if $X^*$ is separable, then for some $\alpha\kle\omega_1$,
  we have $\sz(X)\keq\omega^\alpha$ and the above supremum is not
  attained.
\end{thm}

We next consider the Szlenk index of sums of Banach spaces. For finite
sums (Proposition~\ref{prop:szlenk-sum}) this can be computed directly
from the definition using a kind of ``Leibniz'' rule for higher order
derivatives of products of $w^*$-compact sets
(see Lemma~\ref{lem:szlenk-leibnitz}, part~(iii)). For infinite sums
(Proposition~\ref{prop:szlenk-ell2-sum}) we use the weak index as well
as the result on finite sums to obtain an upper bound. In what comes
we will denote by $\alpha\boxplus\beta$ the ``pointwise sum'' of
ordinals $\alpha$ and $\beta$, which is defined as follows. We first
write $\alpha$ and $\beta$ in Cantor Normal Form as
\[
\alpha=\omega^{\gamma_1}\kcdot m_1 +\omega^{\gamma_2}\kcdot m_2 +
\dots + \omega^{\gamma_k}\kcdot m_k\quad\text{and}\quad
\beta=\omega^{\gamma_1}\kcdot n_1 +\omega^{\gamma_2}\kcdot n_2 + \dots
+ \omega^{\gamma_k}\kcdot n_k\ ,
\]
where $k\kin\bn$, $\gamma_1\kge\gamma_2\kge\dots\kge\gamma_k\kgeq 0$
are ordinals, and $m_i, n_i\kle\omega$ for all~$i$. Then we set
\[
\alpha\boxplus\beta = \omega^{\gamma_1}\kcdot (m_1\kplus n_1)
+\omega^{\gamma_2}\kcdot (m_2\kplus n_2) + \dots +
\omega^{\gamma_k}\kcdot (m_k\kplus n_k)\ .
\]
\begin{lem}
  \label{lem:szlenk-leibnitz}
  Let $X$ and $Y$ be separable Banach spaces. Let $m\kin\bn$, let
  $K,K_1,\dots,K_m$ be non-empty, $w^*$-compact subsets of $X^*$,
  and let $L$ be a non-empty, $w^*$-compact subset of $Y^*$. Let
  $\ve\kge 0$.
  \begin{mylist}{(iii)}
  \item
    $\ds \Big( \bigcup_{j=1}^m K_j\Big) '_\ve \subset \bigcup_{j=1}^m
    \big(K_j\big)'_{\ve/2}\ .$
  \item
    For the subset $K\ktimes L$ of $X^*\oplus_\infty Y^*\keq
    (X\oplus_1 Y)^*$ we have
    \[
    (K\ktimes L)'_\ve = K'_\ve\ktimes L \cup K\ktimes L'_\ve\ .
    \]
  \item
    For the subset $K\ktimes L$ of $X^*\oplus_\infty Y^*$, and for any
    ordinal $\alpha\kle\omega_1$ we have
    \[
    (K\ktimes L)^{(\alpha)}_\ve \subset\bigcup
    _{\beta\boxplus\gamma=\alpha}K^{(\beta)}_{\ve/2}\ktimes
    L^{(\gamma)}_{\ve/2}\ .
    \]
  \end{mylist}
\end{lem}
\begin{proof}
  (i)~Let $x^*\kin\big( \bigcup_{j=1}^m K_j\big) '_\ve$. Since the
  $w^*$-topology on a $w^*$-compact subset of $X^*$ is metrizable, it
  follows easily from the definition that there is a sequence
  $(x^*_n)$ in $\bigcup_{j=1}^m K_j$ such that
  $x^*_n\stackrel{w^*}{\to} x^*$ as $n\to\infty$, and
  $\norm{x^*_n\kminus x^*}\kge\ve/2$ for all $n\kin\bn$. After passing
  to a subsequence we may assume that for some $1\kleq j\kleq m$ we
  have $x^*_n\kin K_j$, and hence $x^*\kin \big(K_j\big)'_{\ve/2}$.

  \noindent
  (ii)~This is immediate from the definition and from the fact that we
  are working with the $\ell_\infty$-sum of $X^*$ and $Y^*$.

  \noindent
  (iii)~We prove this statement by induction. The case $\alpha\keq 0$
  is clear. Using parts~(i) and~(ii) and assuming the statement for
  some $\alpha$, we have
  \begin{align*}
    (K\ktimes L)^{(\alpha+1)}_\ve &\subset\bigcup
    _{\beta\boxplus\gamma=\alpha} \Big( K^{(\beta)}_{\ve/2}\ktimes
    L^{(\gamma)}_{\ve/2} \Big)'_{\ve/2}\\
    & = \bigcup _{\beta\boxplus\gamma=\alpha} \Big(
    K^{(\beta+1)}_{\ve/2}\ktimes L^{(\gamma)}_{\ve/2} \cup
    K^{(\beta)}_{\ve/2}\ktimes L^{(\gamma+1)}_{\ve/2} \Big)\\
    &= \bigcup _{\beta\boxplus\gamma=\alpha+1} K^{(\beta)}_{\ve/2}
    \ktimes L^{(\gamma)}_{\ve/2}\ .
  \end{align*}
  Finally, if $\lambda$ is a limit ordinal, then by the
  induction hypothesis we have
  \[
  (K\ktimes L)^{(\lambda)}_\ve \subset\bigcap _{\alpha<\lambda} \Big(
  \bigcup _{\beta\boxplus\gamma=\alpha}K^{(\beta)}_{\ve/2}\ktimes
  L^{(\gamma)}_{\ve/2} \Big)\ .
  \]
  Write $\lambda$ in Cantor Normal Form:
  \[
  \lambda = \omega^{\lambda_1}\kcdot m_1 + \dots +
  \omega^{\lambda_k}\kcdot m_k
  \]
  with $m_k\kge 0$, and for all $n\kin\bn$ set
  \[
  \alpha_n = \omega^{\lambda_1}\kcdot m_1 + \dots +
  \omega^{\lambda_k}\kcdot (m_k\kminus 1) + \omega^{\delta_n}\kcdot
  n\ ,
  \]
  where $\lambda_k\keq\delta_n\kplus1$ for all $n\kin\bn$ if
  $\lambda_k$ is a successor ordinal, and
  $\delta_n\nearrow\lambda_k$ if $\lambda_k$ is a limit ordinal. 

  Now let $x^*\keq(y^*,z^*)\kin (K\ktimes L)^{(\lambda)}_\ve$. For
  each $n\kin\bn$ we have $\alpha_n\kle\lambda$, so there exist
  ordinals $\beta_n$ and $\gamma_n$ with $\beta_n\boxplus \gamma_n\keq
  \alpha_n$ and $x^*\kin K^{(\beta_n)}_{\ve/2}\ktimes
  L^{(\gamma_n)}_{\ve/2}$. Now there exist
  $n_1,\dots,n_k,p_1,\dots,p_k\kle\omega$ and an infinite subset $N$
  of $\bn$ such that
  \begin{align*}
    \beta_n &= \omega^{\lambda_1}\kcdot n_1 + \dots +
  \omega^{\lambda_k}\kcdot n_k + \omega^{\delta_n}\kcdot u_n\\
    \gamma_n &= \omega^{\lambda_1}\kcdot p_1 + \dots +
  \omega^{\lambda_k}\kcdot p_k + \omega^{\delta_n}\kcdot v_n\ ,
  \end{align*}
  where $u_n\kplus v_n\keq n$ for all $n\kin N$. Assume that $\sup_n
  u_n\keq \omega$ (the case $\sup_n v_n\keq\omega$ being similar). Set
  \begin{align*}
     \beta &= \omega^{\lambda_1}\kcdot n_1 + \dots +
  \omega^{\lambda_k}\kcdot (n_k\kplus 1)\\
    \gamma &= \omega^{\lambda_1}\kcdot p_1 + \dots +
  \omega^{\lambda_k}\kcdot p_k\ .
  \end{align*}
  Then $\beta\keq\sup_n \beta_n$, whereas $\gamma\kleq\gamma_n$ for
  all $n\kin N$. It follows that $y^*\kin K^{(\beta)}_{\ve/2}\keq
  \bigcap _n K^{(\beta_n)}_{\ve/2}$, and $z^*\kin
  L^{(\gamma)}_{\ve/2}$. Since
  $\beta\boxplus\gamma\keq\lambda$, statement~(iii) with $\lambda$
  replacing $\alpha$ follows.
\end{proof}

\begin{prop}
  \label{prop:szlenk-sum}
  Let $X$ and $Y$ be separable Banach spaces. Then
  \[
  \sz\big( X\oplus_1 Y\big) = \max\{ \sz(X),\sz(Y)\}\ .
  \]
\end{prop}

\begin{proof}
  The inequality $\sz\big( X\oplus_1 Y\big)\kgeq \max\{
  \sz(X),\sz(Y)\}$ follows immediately from
  Theorem~\ref{thm:szlenk-properties}~(ii). The reverse inequality is
  trivial if either both $X$ and $Y$ are finite-dimensional, or one of
  $X$ and $Y$ has non-separable dual. So we can assume by the last
  part of Theorem~\ref{thm:szlenk=weak} that $\max\{
  \sz(X),\sz(Y)\}\keq \omega^\eta$ for some $0\kle\eta\kle\omega_1$.

  Now let $K\keq B_{X^*},\ L\keq B_{Y^*}$ and set $\alpha\keq
  \omega^\eta\kcdot 2\kplus 1$. Applying part~(iii) of
  Lemma~\ref{lem:szlenk-leibnitz} we obtain
  \[
  \big( B_{(X\oplus_1 Y)^*} \big) ^{(\alpha)}_\ve = \big(
  B_{X^*}\times B_{Y^*} \big)^{(\alpha)}_\ve = \emptyset\ .
  \]
  It follows that $\sz(X\oplus _1 Y)\kle \alpha$, and hence
  $\sz(X\oplus _1 Y)\kleq \omega^\eta$, as required.
\end{proof}

\begin{prop}
  \label{prop:szlenk-ell2-sum}
  Let $(X_n)$ be a sequence of separable Banach spaces. Let $X\keq
  \Big( \bigoplus _n X_n \Big)_{\ell_2}$ be the $\ell_2$-sum of
  $(X_n)$, and let $\alpha$ be a countable ordinal. If $\sz(X_n)\kleq
  \omega^\alpha$ for all $n\kin\bn$, then
  $\sz(X)\kleq\omega^{\alpha+1}$.
\end{prop}

In the proof we shall use the following notation. For $n\kin\bn$ we
denote by $P_n$ the canonical projection of $X$ onto $X_n$, \ie for
$x\keq(x_i)\kin X$ with $x_i\kin X_i$ for all $i\kin\bn$ we have
$P_n(x)\keq x_n$. For a finite subset $A$ of $\bn$ we let $P_A\keq
\sum_{n\in A} P_n$.

\begin{proof}
  Assume for a contradiction that $\sz(X)\kge\omega^{\alpha+1}$. By
  Theorem~\ref{thm:szlenk=weak} there exists $\rho\kin(0,1)$ such
  that setting
  \[
  \cF = \big\{ (x_i)_{i=1}^n\kin S_X^{<\omega}:\,n\kin\bn,\
  \bignorm{\sum a_ix_i}\kgeq\rho \sum a_i\ \V (a_i)_{i=1}^n\ksubset
  \br^+\big\}
  \]
  we have $\wi(\cF)\kge\omega^{\alpha+1}$. Note that by the geometric
  form of the Hahn-Banach theorem we have
  $(x_1,\dots,x_n)\kin S_X^{<\omega}$ belongs to $\cF$ if and only if
  there exists $x^*\kin S_{X^*}$ such that $x^*(x_i)\kgeq\rho$ for
  each $i\keq 1,\dots,n$.

  We are going to show the following claim. Let $\vx\keq
  (x_1,\dots,x_m)\kin S_X^{<\omega}$, and let $k\kgeq 1,\ M\kgeq 0$ be
  integers. Assume that $\wi\big(\cF(\vx)\big) \kge\omega^\alpha
  \kcdot k$. Then there exists $\vy\keq (y_1,\dots,y_n)\kin
  S_X^{<\omega}$ such that
  \begin{mylist}{(ii)}
  \item
    $\wi \big( \cF (\vx,\vy)\big) >\omega^\alpha\kcdot (k\kminus 1)$,
    and
  \item
    for all $x^*\kin S_{X^*}$ there exists $i\kin\{1,\dots,n\}$ such
    that $x^*\big(P_{[1,M]}(y_i)\big)\kle \rho/4$.
  \end{mylist}
  Let us first see how this claim completes the proof. Fix $K\kin\bn$
  with $K\kge 4/\rho^2$.
  We obtain sequences $\vx_1, \dots, \vx_K$ in $S_X^{<\omega}$ and
  $N_1\kle \dots \kle N_K$ in $\bn$ recursively as follows: at the
  $j^{\text{th}}$ step we apply the claim with $\vx\keq
  (\vx_1,\dots,\vx_{j-1}),\ k\keq K\kminus j\kplus 1$ and $ M\keq
  N_{j-1}$ to obtain $\vy$ as above (for $j\keq 1$ we begin with
  $\vx\keq\emptyset,\ k\keq K$ and $M\keq N_0\keq 0$). Then we set
  $\vx_j\keq \vy$ and choose $N_j\kge N_{j-1}$ such that writing
  $\vx_j\keq (y_{j,1},\dots,y_{j,L_j})$ we have
  \[
  \norm{y_{j,\ell}-P_{[1,N_j]}(y_{j,\ell})}<\rho/4\qquad\text{for
  }\ell\keq 1,\dots,L_j\ .
  \]
  From property~(i) we get in particular that
  $(\vx_1,\dots,\vx_K)\kin\cF$. Thus there exists $x^*\kin S_{X^*}$
  such that
  \[
  x^*(y_{j,\ell})\geq\rho\qquad\text{for all }j\keq 1,\dots,K,\
  \ell\keq 1,\dots, L_j\ .
  \]
  It follows that for each $j\keq 1,\dots,K$ we find $1\kleq
  \ell_j\kleq L_j$ such that
  \[
  x^*\big( P_{[N_{j-1}+1,N_j]} (y_{j,\ell_j}) \big) > \rho/2\ ,
  \]
  and hence we get $\norm{x^*}\kgeq \frac{\sqrt{K}\rho}2\kge 1$, which
  is a contradiction.

  We now turn to the proof of the claim. Define
  \begin{multline*}
    \cG =\Big\{ (y_1,\dots,y_n)\kin \cF(\vx):\, n\kin\bn,\ \E x^*\kin
    S_{X^*},\\
    x^*\big(P_{[1,M]}(y_i)\big) \kgeq\rho/4 \text{ for }i\keq
    1,\dots,n\Big\}\ .
  \end{multline*}
  Note that $\cG$ is a tree on $S_X^{<\omega}$. If
  $\wi(\cG)\kge\omega^\alpha$, then by Theorem~\ref{thm:szlenk=weak}
  we have $\sz(X_1\oplus_2\dots\oplus_2 X_M)\kge \omega^\alpha$
  contradicting Proposition~\ref{prop:szlenk-sum}. So we have
  $\wi(\cG)\kleq \omega^\alpha$. On the other hand, we have
  \[
  \wi \Big( \big( \cF(\vx) \big)\wder{\omega^\alpha\cdot (k-1)} \Big)
  > \omega^\alpha\ .
  \]
  Thus we can find
  \[
  \vy\kin \big( \cF(\vx) \big)\wder{\omega^\alpha\cdot (k-1)}
  \setminus \cG\ .
  \]
  Properties (i) and (ii) are now easily checked.
\end{proof}

\begin{prop}
  \label{prop:tsirelson-szlenk}
  Let $\alpha$ be an ordinal with $1\kleq\alpha\kle\omega_1$. The
  Szlenk-index of the Tsirelson space of order $\alpha$ is given by
  \[
  \sz(T_\alpha) = \omega ^{\alpha\cdot\omega}\ .
  \]
\end{prop}

\begin{proof}
  For each $\rho\kin (0,1)$ let 
  \[
  \cF_\rho = \bigg\{ (x_i)_{i=1}^n\kin
  S_{T_\alpha}^{<\omega}:\,n\kin\bn,\ \Bignorm{\sum_{i=1}^n
  a_ix_i}\kgeq\rho \sum_{i=1}^n a_i\ \V (a_i)_{i=1}^n\ksubset
  \br^+\bigg\}\ .
  \]
  We first show that $\sz(T_\alpha)\kgeq \omega
  ^{\alpha\cdot\omega}$. Let $(e_i)$ be the unit vector basis of
  $T_\alpha$. It follows from the definition of $T_\alpha$ that for
  each $n\kin\bn$ if $F\kin \cS_{\alpha\cdot n}$, then $(e_i)_{i\in
  F}$ is $2^n$-equvalent to the unit vector basis of
  $\ell_1^{\abs{F}}$, so in particular $(e_i)_{i\in F}\kin
  \cF_{2^{-n}}$. Hence by Proposition~\ref{prop:S-index-fine-schreier}
  and Theorem~\ref{thm:szlenk=weak} we have
  \[
  \omega^{\alpha\cdot n} < \wi(\cF_{2^{-n}}) < \sz(T_\alpha)\ .
  \]
  Since $n\kin\bn$ was arbitrary, the inequality $\sz(T_\alpha)\kgeq
  \omega ^{\alpha\cdot\omega}$ follows at once.

  For the reverse inequality assume that
  $\omega^\gamma\kle\sz(T_\alpha)$, where
  $\gamma\keq\alpha\kcdot\omega$. Then by
  Theorem~\ref{thm:szlenk=weak} there exists $\rho\kin (0,1)$ with
  $\omega^\gamma\kle\wi(\cF_\rho)$, and by
  Proposition~\ref{prop:S-index-fine-schreier} there is a family
  $\cG\keq \big(x_F\big) _{F\in\cS_\gamma\setminus\{\emptyset\}}
  \ksubset\cF_\rho$ such that for all $F\kin \cS_\gamma\ksetminus
  \MAX(\cS_\gamma)$ the sequence $\big(x_{F\cup\{n\}}\big) _{n>\max
  F}$ is weakly null.

  By standard perturbation arguments we may, after making $\rho$
  smaller and replacing $(x_F)
  _{F\in\cS_\gamma\setminus\{\emptyset\}}$ by $(x_{F'})
  _{F\in\cS_\gamma\setminus\{\emptyset\}}$ for an appropriate pruning
  function $F\mapsto F'\colon \cS_\gamma\to\cS_\gamma$, if necessary,
  assume that $\cG$ is a block tree of $(e_i)$ in $T_\alpha$.

  We now apply a result of R.~Judd and the first named
  author~\cite{JuO}. In their
  terminology $\cG$ is an \emph{$\ell_1$-$K$-block basis tree} of
  $(e_i)$ in $T_\alpha$ with $K\keq \rho^{-1}$ (we are using the
  $1$-unconditionality of $(e_i)$), and hence its order o$(\cG)$ is at
  most the Bourgain $\ell_1$-index of $T_\alpha$ which is shown to be
  $\omega^\gamma$ in~\cite{JuO}. On the other hand, by
  Proposition~\ref{prop:S-index-fine-schreier} we have
  o$(\cG)\kge\omega^\gamma$. This contradiction completes the proof.
\end{proof}

\begin{prop}
  \label{prop:upper-tsirelson-upper-szlenk}
  Let $\alpha$ be a countable ordinal and let $Z$ be a Banach space
  with an FDD $E\keq(E_i)$ that satisfies subsequential
  $T_\alpha$-upper estimates. Then $\sz(Z)\kleq
  \omega^{\alpha\cdot\omega}$.
\end{prop}

\begin{proof}
  For each $\rho\kin(0,1)$ let
  \[
  \cF_\rho = \bigg\{ (x_i)_{i=1}^n\kin S_Z^{<\omega}:\,n\kin\bn,\
  \Bignorm{\sum_{i=1}^n a_ix_i}\kgeq\rho \sum_{i=1}^n a_i\ \V
  (a_i)_{i=1}^n\ksubset \br^+\bigg\}\ .
  \]
  and let
  \[
  \cG_\rho = \bigg\{ (x_i)_{i=1}^n\kin
  S_{T_\alpha}^{<\omega}:\,n\kin\bn,\ \Bignorm{\sum_{i=1}^n
    a_ix_i}\kgeq\rho \sum_{i=1}^n a_i\ \V (a_i)_{i=1}^n\ksubset
  \br^+\bigg\}\ .
  \]
  Fix $\gamma\kle\sz(Z)$. By Theorem~\ref{thm:szlenk=weak} there
  exists $\rho\kin (0,1)$ such that $\wi(\cF_\rho)\kge\gamma$, and by
  Proposition~\ref{prop:S-index-fine-schreier} there is a family
  $(x_F) _{F\in\cF_\gamma\setminus\{\emptyset\}}\ksubset \cF_{\rho}$
  such that for all $F\kin \cF_\gamma\ksetminus \MAX(\cF_\gamma)$ the
  sequence $\big(x_{F\cup\{n\}}\big) _{n>\max F}$ is weakly null.

  By standard perturbation arguments we may, after making $\rho$
  smaller and appropriately pruning $(x_F)
  _{F\in\cF_\gamma\setminus\{\emptyset\}}$, if necessary, assume that
  $(x_F) _{F\in\cF_\gamma\setminus\{\emptyset\}}$ is a block tree of
  $(E_i)$ in $Z$, and that for all $F\kin \cF_\gamma\ksetminus
  \MAX(\cF_\gamma)$ the sequence $\big(x_{F\cup\{n\}}\big) _{n>\max
  F}$ is a block basis of $(E_i)$. Let $(e_i)$ be the unit vector
  basis of $T_\alpha$ and define
  \[
  t_F = e_{\min\supp_E(x_F)}\qquad\text{for all
  }F\kin\cF_\gamma\setminus\{\emptyset\}\ .
  \]
  Note that $(t_F)_{F\in\cF_\gamma\setminus\{\emptyset\}}$ is a block
  tree of $(e_i)$ in $T_\alpha$ and that it is contained in
  $\cG_{\rho'}$ for some $\rho'\kin(0,1)$ since  $(E_i)$ that
  satisfies subsequential $T_\alpha$-upper estimates. Since $(e_i)$ is
  shrinking, it follows by
  Proposition~\ref{prop:S-index-fine-schreier} that
  $\gamma\kle\wi(\cG_{\rho'})$. Using Theorem~\ref{thm:szlenk=weak}
  and Proposition~\ref{prop:tsirelson-szlenk} we deduce that
  $\sz(Z)\kleq \omega^{\alpha\cdot\omega}$, as required.
\end{proof}

\begin{rem}
  It follows from properties of higher order Tsirelson spaces
  (Proposition~\ref{prop:tsirelson-properties}) that the unit vector
  basis of $\di{T}{*}{\alpha}$ satisfies subsequential
  $T_\alpha$-upper estimates. Hence the above result shows that
  \[
  T_\alpha \in \cC_{\omega^{\alpha\cdot\omega}} = \big\{ X:\,X\text{
  is separable, reflexive,
  }\max\{\sz(X),\sz(X^*)\}\kleq{\omega^{\alpha\cdot\omega}} \big\}\ .
  \]
\end{rem}

\section{The main theorem and its consequences}
\label{section:main}

\begin{thm}
  \label{thm:main}
  Let $Z$ be a Banach space with a shrinking, bimonotone FDD $(E_i)$
  and let $X$ be an infinite-dimensional closed subspace of $Z$. Then
  for any $C\kge 4$ there exist an ordinal $\alpha\kle\sz(X)$, a
  sequence
  $\deltab\keq(\delta_i)\ksubset(0,1)$ with $\delta_i\downarrow 0$,
  and a blocking $(G_i)$ of $(E_i)$ with $G_i\keq \bigoplus
  _{j=m_{i-1}}^{m_i-1} E_j,\ i\kin\bn,\ 1\keq m_0\kle m_1\kle
  m_2\kle\dots$, such that if $(x_i)\ksubset S_X$ is a $\deltab$-block
  sequence of $(G_n)$ with $\norm{x_i\kminus
  P^G_{(s_{i-1},s_i]}x_i}\kle\delta_i$ for all $i\kin\bn$, $1\kleq
  s_0\kle s_1\kle s_2\kle\dots$, then $(x_i)$ is $C$-dominated by
  $(e_{m_{s_{i-1}}})$, where $(e_i)$ is the unit vector basis of
  $T_{\cF_\alpha,\frac12}$.
\end{thm}

We first prove some consequences of Theorem~\ref{thm:main}. In
Corollary~\ref{cor:szlenk-lower-upper-est} below we recast the
property of being in the class $\cC_\alpha$ in terms of certain lower
and upper Tsirelson-norm estimates. In Theorem~\ref{thm:tight-ta} we
show that for certain values of $\alpha$ these estimates are best
possible, which proves Theorem~\ref{mainthm:structural} from the
Introduction. These norm estimates and Theorem~\ref{thm:universal} are
the two main ingredients in answering Pe\l czy\'nski's question which
we do in Theorem~\ref{thm:universal-ca} followed by a refinement in
Theorem~\ref{thm:universal-ca2}. We then state a result of Johnson
which we use to deduce basis versions
(Theorems~\ref{mainthm:embedding} and~\ref{mainthm:universal}) of
Theorem~\ref{thm:universal-ca2}. The rest of the section is taken up
by the proof of Theorem~\ref{thm:main}.

\begin{cor}
  \label{cor:szlenk-upper-est}
  Let $X$ be an infinite-dimensional Banach space with separable
  dual. There exists an ordinal $\alpha\kle \sz(X)$ such that $X$
  satisfies subsequential $C$-$T_{\cF_\gamma,\frac12}$-upper tree
  estimates for any ordinal $\gamma\kgeq\alpha$, where $C$ is a
  universal constant.
\end{cor}

\begin{proof}
  By Zippin's theorem~\cite{Zippin:88}, $X$ $K$-embeds into a Banach space $Z$
  with a shrinking, bimonotone FDD $(E_i)$, where $K$ is a universal
  constant. Renorming $X$ with a $K$-equivalent norm we may assume
  without loss of generality that $X$ is a subspace of $Z$. We now
  apply Theorem~\ref{thm:main} to obtain $\alpha\kle\sz(X)$, a
  sequence $\deltab\keq(\delta_i)\ksubset(0,1),\ \delta_i\downarrow
  0$, and a blocking $(G_i)$ of $(E_i)$ with $G_i\keq \bigoplus
  _{j=m_{i-1}}^{m_i-1} E_j,\ i\kin\bn,\ 1\keq m_0\kle m_1\kle
  m_2\kle\dots$, such that if $(x_i)\ksubset S_X$ is a $\deltab$-block
  sequence of $(G_n)$ with $\norm{x_i\kminus
  P^G_{(s_{i-1},s_i]}x_i}\kle\delta_i$ for all $i\kin\bn$, $1\kleq
  s_0\kle s_1\kle s_2\kle\dots$, then $(x_i)$ is $5$-dominated by
  $(e_{\alpha,m_{s_{i-1}}})$, where $(e_{\alpha,i})$ is the unit vector
  basis of $T_{\cF_\alpha,\frac12}$.

  Fix an ordinal $\gamma\kgeq\alpha$ and an integer $\ell$ such that
  $(e_{\alpha,i})_{i\geq\ell}$ is $1$-dominated by
  $(e_{\gamma,i})_{i\geq\ell}$ (such an integer exists
  by property~\eqref{eq:fine-schreier-increasing} of the fine Schreier
  families). We now show that $X$ satisfies
  subsequential $C$-$T_{\cF_\gamma,\frac12}$-upper tree estimates with
  $C\keq 5$. Let $(x_t)_{t\in\odd}$ be a normalized, weakly null even
  tree in $X$. We will inductively choose sequences $s_0\kle
  s_1\kle\dots$ and $n_1\kle n_2\kle\dots$ in $\bn$ as follows. Set
  $s_0\keq 1$ and $n_1\keq\max(\ell,m_1)$. Assume that for some
  $i\kin\bn$ we have already chosen $s_0\kle s_1\kle\dots\kle s_{i-1}$
  and $n_1\kle n_2\kle\dots\kle n_{2i-1}$. Since nodes are weakly null,
  there exists $n_{2i}\kge n_{2i-1}$ such that
  \[
  \bignorm{P^G_{[1,s_{i-1}]}x_{(n_1,n_2,\dots,n_{2i})}} <\delta_i\ .
  \]
  Then choose $s_i\kge s_{i-1}$ such that
  \[
  \bignorm{x_{(n_1,n_2,\dots,n_{2i})} -
  P^G_{(s_{i-1},s_i]}x_{(n_1,n_2,\dots,n_{2i})}} <\delta_i\ .
  \]
  Finally, choose $n_{2i+1}\kge n_{2i}$ with $n_{2i+1}\kgeq
  m_{s_i}$. This completes the recursive construction. It follows
  immediately from the choice of $\alpha, \deltab, (G_i)$ and $\ell$,
  and from the $1$-right-dominant property of $(e_{\alpha,i})$ that
  $\big(x_{(n_1,n_2,\dots,n_{2i})}\big)$ is $5$-dominated by
  $(e_{\gamma,n_{2i-1}})$.
\end{proof}

\begin{cor}
  \label{cor:szlenk-lower-upper-est}
  Let $X$ be an infinite-dimensional, separable, reflexive Banach
  space. Then
  there exists an ordinal $\gamma\kle\max\{\sz(X),\sz(X^*)\}$  such that
  $X$ satisfies subsequential
  $C$-$\big(\di{T}{*}{\cF_\delta,\frac12},
  T_{\cF_\delta,\frac12}\big)$ tree estimates for any
  $\delta\kgeq\gamma$, where $C$ is a universal constant.
\end{cor}

\begin{proof}
  By Corollary~\ref{cor:szlenk-upper-est} there is a universal
  constant $C$, and there exist ordinals $\alpha\kle\sz(X)$ and
  $\beta\kle\sz(X^*)$ such that $X$ satisfies subsequential
  $C$-$T_{\cF_\gamma,\frac12}$-upper tree estimates for any
  $\gamma\kgeq\alpha$, and $X^*$ satisfies subsequential
  $C$-$T_{\cF_\delta,\frac12}$-upper tree estimates for any
  $\delta\kgeq\beta$. It follows from
  Proposition~\ref{prop:tree-est-duality} that $X$ satisfies
  subsequential
  $(2C\kplus\ve)$-$\big(\di{T}{*}{\cF_\delta,\frac12},
  T_{\cF_\delta,\frac12}\big)$ tree estimates for any $\ve\kge 0$ and
  for any $\delta\kgeq\max\{\alpha,\beta\}$.
\end{proof}

The above results show that higher order Tsirelson spaces are more
than just mere examples in the hierarchy
$(\cC_\alpha)_{\alpha<\omega_1}$. Indeed they are intimately related
to the Szlenk index of an arbitrary separable, reflexive space and its
dual. The next theorem shows that this relationship is tight in the
classes $\cC_{\omega^{\alpha\cdot\omega}}$: Tsirelson spaces of order
$\alpha$ and their duals are maximal and, respectively, minimal in
these classes. In particular, this proves
Theorem~\ref{mainthm:structural} stated in the Introduction. The proof
uses some further results from~\cite{OSZ2} which we shall not state
here as Theorem~\ref{thm:tight-ta} will not be used in the proof of
our universality results.

\begin{thm}
  \label{thm:tight-ta}
  Let $\alpha\kle\omega_1$. For a separable, reflexive space $X$ the
  following are equivalent.
  \begin{mylist}{(iii)}
  \item
    $X\kin\cC_{\omega^{\alpha\cdot\omega}}$.
  \item
    $X$ satisfies subsequential $\big(\di{T}{*}{\alpha,c},
    T_{\alpha,c}\big)$ tree estimates for some $c\kin (0,1)$.
  \item
    $X$ embeds into a separable, reflexive space $Z$ with an FDD
    $(E_i)$ which satisfies subsequential
    $(\di{T}{*}{\alpha,c},T_{\alpha,c})$ estimates in $Z$ for some
    $c\kin (0,1)$.
  \end{mylist}
\end{thm}

\begin{proof}
  ``(i)$\Rightarrow$(ii)'' By
  Corollary~\ref{cor:szlenk-lower-upper-est} there exists
  $n\kle\omega$ such that $X$ satisfies subsequential
  $\big(\di{T}{*}{\alpha\cdot n}, T_{\alpha\cdot n}\big)$ tree
  estimates. It is not hard to show directly from the definition that
  the norms $\norm{\cdot}_{T_{\alpha\cdot n}}$ and
  $\norm{\cdot}_{T_{\alpha,c}}$ on $\coo$, where $c\keq
  \frac1{2^{1/n}}$, are equivalent. Hence~(ii) follows.

  \noindent
  ``(ii)$\Rightarrow$(iii)'' This is immediate from~\cite[Theorem
  15]{OSZ2}. We note that ``(iii)$\Rightarrow$(ii)''
  is straightforward from the definition.

  \noindent
  ``(iii)$\Rightarrow$(i)'' Let $Z$ be the space given by~(iii), and
  choose $n\kin\bn$ such that $c^n\kleq\frac12$. It follows directly
  from the definition that the unit vector basis of $T_{\alpha,c}$ is
  dominated by the unit vector basis of $T_{\alpha\cdot n,c^n}$ which
  in turn is dominated by the unit vector basis of $T_{\alpha\cdot
  n}$. Hence by Theorem~\ref{thm:szlenk-properties}(ii), and by
  Propositions~\ref{prop:upper-tsirelson-upper-szlenk}
  and~\ref{prop:tsirelson-szlenk} we have
  \[
  \sz(X) \leq \sz(Z) \leq \sz(T_{\alpha\cdot n}) = \omega^{\alpha\cdot
  n\cdot \omega} = \omega^{\alpha\cdot\omega}\ .
  \]
  (Alternatively, one can just observe that the proof of
  Proposition~\ref{prop:tsirelson-szlenk} works for $T_{\alpha,c}$,
  \ie we have $\sz(T_{\alpha,c})\keq \omega^{\alpha\cdot\omega}$.) Now
  since $X$ satisfies~(ii), it follows from duality
  (Proposition~\ref{prop:tree-est-duality} and~\cite[Corollary~14]
  {OSZ2}) that~(ii), and hence~(iii), also hold with $X$ replaced by
  $X^*$. This gives $\sz(X^*)\kleq\omega^{\alpha\cdot\omega}$. Thus
  $X\kin\cC_{\omega^{\alpha\cdot\omega}}$, as required.
\end{proof}

\begin{rem}
  Using the proof of Proposition~\ref{prop:tsirelson-szlenk} one can
  show that $\sz(T_{\cF_\alpha,c})\keq\alpha^{\omega}$ whenever
  $1\kleq\alpha\kle\omega_1$ and $c\kin(0,1)$. Considering the Cantor
  Normal Form of $\alpha$, it is possible to write
  $\alpha^\omega\keq\omega^{\beta\cdot\omega}$ for some
  $\beta\kleq\alpha$. Thus, it is not possible to obtain a finer
  gradiation of the hierarchy $(\cC_\alpha)_{\alpha<\omega_1}$ by
  using fine Schreier families.
\end{rem}

We are now in the position to answer Pe\l czy\'nski's question. We
shall use the notation
\[
\cA_\alpha(C)\keq \cA_{\di{T}{*}{\alpha},
  T_\alpha}(C)\qquad\text{and}\qquad
\cA_\alpha\keq\bigcup _{C<\infty} \cA_\alpha(C)\ ,
\]
where $0\kle\alpha\kle\omega_1$ and $C\kin[1,\infty)$ (see also the
notation preceding Theorem~\ref{thm:universal}). Recall that
\[
\cC_\alpha = \big\{ X:\,X\text{ is separable,
  reflexive, }\max\{\sz(X),\sz(X^*)\}\kleq\alpha \big\}\ ,
\]
and that the Szlenk index of an infinite-dimensional Banach space with
separable dual is of the form $\omega^\eta$ for some
$0\kle\eta\kle\omega_1$ (Theorem~\ref{thm:szlenk=weak}), so we need
only consider the classes $\cC_\alpha$ when $\alpha$ is of this
form. We should also comment on finite-dimensional spaces before
proceeding.

For $\alpha\kle\omega$ we have $\cC_\alpha\keq\cC_0$ is the class of
all finite-dimensional spaces. Let $Z$ be the $\ell_2$-sum of a
countable, dense (with respect to the Banach-Mazur distance) subset of
$\cC_0$. Then by Proposition~\ref{prop:szlenk-ell2-sum} we have
$Z\kin\cC_\omega$. Moreover, $Z$ is universal for $\cC_0$: for all
$X\kin\cC_0$ and for all $\ve\kge 0$, $X$ $(1\kplus\ve)$-embeds into
$Z$.

For $\omega\kleq\alpha\kle\omega_1$ we have $\ell_2\kin C_\alpha$, and
hence $\ell_2\oplus X\kin \cC_\alpha$ for any finite-dimensional space
$X$. Thus we can restrict attention to infinite-dimensional spaces for
the purpose of finding a universal space for the class $\cC_\alpha$.

\begin{thm}
  \label{thm:universal-ca}
  For every ordinal $\alpha$ with $0\kle\alpha\kle\omega_1$ there is a
  separable, reflexive space with an FDD which is universal for the
  class $\cC_{\omega^\alpha}$.

  More precisely, there is a universal constant $K$, such that for all
  $0\kle\alpha\kle\omega_1$ there exists a space
  $Z\kin\cC_{\omega^{\alpha\cdot\omega}}$ with an FDD such that every
  space $X\kin\cC_{\omega^\alpha}$ $K$-embeds into $Z$.
\end{thm}

\begin{proof}
  Let $C\kin [1,\infty)$ be the universal constant of
  Corollary~\ref{cor:szlenk-lower-upper-est}, and let $B,
  D\kin[1,\infty)$ be the universal constants of
  Proposition~\ref{prop:tsirelson-properties}. Let $K\keq
  K_{B,D,1,1}(C)$ be the constant from
  Theorem~\ref{thm:universal}. Given $0\kle\alpha\kle\omega_1$, let
  $Z\kin \cA_{\alpha}$ be the universal space given by
  Theorem~\ref{thm:universal} with $U\keq T_\alpha$ and $V\keq
  U^*$. In particular $Z$ has an FDD $(E_i)$ that satisfies
  subsequential $(\di{T}{*}{\alpha},T_\alpha)$ estimates in $Z$. By an
  easy duality argument the FDD $(E^*_i)$ of $Z^*$ satisfies
  subsequential $(\di{T}{*}{\alpha},T_\alpha)$ estimates in
  $Z^*$. Hence by Proposition~\ref{prop:upper-tsirelson-upper-szlenk}
  we have $\max\{ \sz(Z),\sz(Z^*)\}\kleq\omega^{\alpha\cdot\omega}$,
  \ie $Z\kin \cC_{\omega^{\alpha\cdot \omega}}$.

  Now let $X\kin \cC_{\omega^\alpha}$ be an infinite-dimensional
  space. By Corollary~\ref{cor:szlenk-lower-upper-est} we have $X\kin
  \cA_\alpha(C)$, and hence $X$ $K$-embeds into $Z$.
\end{proof}

\begin{rem}
  By a result of Johnson and Odell~\cite{JO:05}, the space $Z$
  constructed in the proof of
  Theorem~\ref{thm:universal-ca} cannot be in the class
  $\cC_{\omega^\alpha}$. Indeed, if that was the case, then every
  space that embeds into $Z$ would in fact $K$-embed into $Z$. Such a
  space is called \emph{elastic }in~\cite{JO:05}, where it is proved that
  a separable, elastic space contains $\co$. Obviously, $Z$ cannot
  contain $\co$ giving the required contradiction.
\end{rem}

Note that the above theorem yields a universal space for the class
$\cC_{\omega^{\alpha\cdot\omega}}$ that lives in the class
$\cC_{\omega^{\alpha\cdot{\omega^2}}}$. A small modification of the
proof gives the slightly better result mentioned in the Introduction:

\begin{thm}
  \label{thm:universal-ca2}
  For every $\alpha\kle\omega_1$ there is a space
  $Z_\alpha\kin\cC_{\omega^{\alpha\cdot\omega+1}}$ with an FDD which
  is universal for the class $\cC_{\omega^{\alpha\cdot\omega}}$.

  More precisely, there is a universal constant $K$, and for each
  $\alpha\kle\omega_1$ there is a sequence
  $\big(Z_{\alpha,n}\big)_{n=1}^\infty$ of spaces with FDDs in
  $\cC_{\omega^{\alpha\cdot\omega}}$ such that for all $X\kin
  \cC_{\omega^{\alpha\cdot\omega}}$ there exists $n\kin\bn$ such that
  $X$ $K$-embeds into $Z_{\alpha,n}$. The space $Z_\alpha$ can then be
  taken to be the $\ell_2$-direct sum of the sequence
  $\big(Z_{\alpha,n}\big)_{n=1}^\infty$.
\end{thm}

\begin{proof}
  For $\alpha\keq 0$ we have already done this just before stating
  Theorem~\ref{thm:universal-ca}. Now assume that
  $0\kle\alpha\kle\omega_1$, and let $C,K$ be the constants defined
  in the proof of
  Theorem~\ref{thm:universal-ca}. Let
  $Z_{\alpha,n}\kin \cA_{\alpha\cdot n}$ be the universal space given
  by Theorem~\ref{thm:universal} with $U\keq T_{\alpha\cdot n}$ and
  $V\keq U^*$. As in the proof of Theorem~\ref{thm:universal-ca} we
  deduce that $\max\{ \sz(Z_{\alpha,n}), \sz(\di{Z}{*}{\alpha,n})\}
  \kleq \omega^{\alpha\cdot n\cdot\omega}\keq
  \omega^{\alpha\cdot\omega}$, \ie that
  $Z_{\alpha,n}\kin\cC_{\omega^{\alpha\cdot\omega}}$. Now let
  $Z_\alpha\keq \Big(\bigoplus_{n=1}^\infty
  Z_{\alpha,n}\Big)_{\ell_2}$
  be the $\ell_2$-direct sum of the sequence
  $\big(Z_{\alpha,n}\big)_{n=1}^\infty$. By
  Proposition~\ref{prop:szlenk-ell2-sum} we have $Z_\alpha\kin
  \cC_{\omega^{\alpha\cdot\omega+1}}$.

  Finally, let $X\kin \cC_{\omega^{\alpha\cdot\omega}}$ be an
  infinite-dimensional space. By
  Corollary~\ref{cor:szlenk-lower-upper-est} there exists $n\kin\bn$
  such that $X\kin \cA_{\alpha\cdot n}(C)$, and hence $X$ $K$-embeds
  into $Z_{\alpha,n}$ and into $Z_\alpha$.
\end{proof}

As indicated in the Introduction, Theorem~\ref{mainthm:embedding} and
Theorem~\ref{mainthm:universal} follow now from
Theorem~\ref{thm:universal-ca2} by applying the following result of
Johnson~\cite{Johnson:71}.

\begin{thm}[{\cite[Theorem A]{Johnson:71}}]
  \label{thm:johnson}
  Let $(G_i)$ be a sequence of finite-dimensional Banach spaces so
  that
  \begin{mylist}{(ii)}
  \item
    if $E$ is a finite-dimensional Banach space and $\ve\kge 0$, then
    there is an $i\kin\bn$ so that $d(E,G_i)\keq\inf\{
    \norm{T}\kcdot\norm{T^{-1}}:\,T\colon E\to G_i\text{ is an
      isomorphism}\}\kle 1\kplus\ve$,
  \item
    for each $i\kin\bn$ there is an infinite $J\ksubset \bn$ so that
    $G_i$ and $G_j$ are isometric for all $j\kin J$.
  \end{mylist}
  Let $C_2\keq(\bigoplus_{i=1}^\infty G_i)_{\ell_2}$ and let $X$ be
  any separable space which has the $\lambda$-metric approximation
  property for some $\lambda\ge 1$.  Then $X\oplus C_2$ has a basis.
\end{thm} 

Note that the $\lambda$-metric approximation property is also known as
the $\lambda$-bounded approximation property.

\begin{proof}[Proof of Theorems~\ref{mainthm:embedding}
    and~\ref{mainthm:universal}]
  Clearly, spaces $X$ with an FDD have the $\lambda$-metric
  approximation property for some $\lambda\kgeq 1$, meaning that for
  any compact set $K\ksubset X$ and $\ve\kge 0$ there is a finite rank
  operator $T$ with $\norm{T(x)-x}\kle \ve$ for all $x\kin K$. Let
  $C_2$ be the space defined in Theorem~\ref{thm:johnson}, and let
  $Z_\alpha$ and $Z_{\alpha,n},\ n\kin\bn$, be the spaces from
  Theorem~\ref{thm:universal-ca2}. Then $Z_\alpha\oplus C_2$ and
  $Z_{\alpha,n}\oplus C_2$ have Schauder bases and it follows from
  Propositions~\ref{prop:szlenk-sum} and~\ref{prop:szlenk-ell2-sum}
  that $\sz(Z_\alpha\oplus
  C_2)\keq\sz(Z_\alpha)\keq\omega^{\alpha\cdot\omega+1}$ and
  $\sz(Z_{\alpha,n}\oplus C_2)\keq\omega^{\alpha\cdot\omega}$.
\end{proof}

In the remainder of this section we give a proof of our main result,
Theorem~\ref{thm:main}, which is at the heart of our embedding and
universality results.

\begin{proof}[Proof of Theorem~\ref{thm:main}]
  Fix a constant $D$ with $4\kle D\kle C$, and choose $\rho\kin (0,1)$
  such that $4\kplus 12\rho D\kle D$. Set
  \[
  \cF\keq \Big\{ (x_i)\kin S_X^{<\omega}:\,\bignorm{\sum a_ix_i}\kgeq
  2\rho\sum a_i \text{ for all }(a_i)\ksubset [0,\infty) \Big\}\ .
  \]
  Note that $\cF$ is a hereditary tree on $S_X^{<\omega}$. Next fix a
  sequence $\veb\keq(\ve_i)\ksubset\big(0,\frac12\big)$ such that
  \[
  \cF^Z_{10\veb} \subset \Big\{ (z_i)\kin
  S_Z^{<\omega}:\,\bignorm{\sum a_iz_i}\kgeq \rho\sum a_i \text{ for
  all }(a_i)\ksubset \br^+ \Big\}\ .
  \]
  Now consider the hereditary block tree $\cG\keq
  \Sigma(E,Z)\cap\cF^Z_{\veb}$ of $(E_i)$ in $Z$ and its
  compression $\cGt$. Let $\alpha$ be the Cantor-Bendixson
  index of $\cGt$. By Proposition~\ref{prop:fat-index} and by
  Theorem~\ref{thm:szlenk=weak} we have
  \[
  \wi\big( \cF^Z_{2\veb} \big) \leq \wi \big(\cF^X_{10\veb}\big) <
  \sz(X)\ .
  \]
  Since $\cG^{E,Z}_{\veb}\ksubset \cF^Z_{2\veb}$, we have $\bli \big(
  \cG^{E,Z}_{\veb} \big) \kleq \wi \big( \cF^Z_{2\veb} \big)$. Since
  $\sz(X)$ is a limit ordinal, it follows by
  Proposition~\ref{prop:compression} that
  \[
  \alpha = \cbi(\cGt) < \sz(X)\ .
  \]
  We now apply Theorem~\ref{thm:thin-ramsey} to obtain an infinite
  subset $M\keq \{m_1, m_2,\dots\}$ of $\bn$ such that
  \begin{equation}
    \label{eq:main;G}
    \MAX (\cF_\alpha)\cap \fin{M}\cap\cGt =\emptyset\ .
  \end{equation}
  To see this, give each element $A$ of the thin family
  $\MAX(\cF_\alpha)$ colour \emph{red }if $A\kin\cGt$, and colour
  \emph{blue }otherwise, and obtain
  $M\keq\{m_1,m_2,\dots\}\kin\infin{\bn}$ such that
  $\MAX(\cF_\alpha)\cap\fin{M}$ is monochromatic. Now the map
  $i\mapsto m_i\colon\bn\to M$ induces a homeomorphism
  $\fin{\bn}\to\fin{M}$ that maps $\cF_\alpha$ onto
  $\cF_\alpha\cap\fin{M}$ (as $\cF_\alpha$ is spreading). Since the
  Cantor-Bendixson index is a topological invariant, it follows that
  $\cbi(\cF_\alpha\cap\fin{M})\keq \alpha\kplus 1$. Hence
  $\MAX(\cF_\alpha)\cap\fin{M}$ cannot be monochromatic \emph{red},
  and thus~\eqref{eq:main;G} follows. Observe that if
  $F\kin\cGt\cap\fin{M}$, then $F\kin\cF_\alpha$.
 
  Without loss of generality we may assume that $m_1\kge 1$. We set
  $m_0\keq 1$ and $G_i\keq \bigoplus _{j=m_{i-1}}^{m_i-1} E_j$ for all
  $i\kin\bn$. Finally, we choose $\deltab\keq(\delta_i)\ksubset(0,1),\
  \delta_i\downarrow 0$, such that
  \begin{align*}
    4\sum _{i=1}^\infty\delta_i &< \min (\rho,C\kminus
    D), &&\text{and}\\
    4\sum_{j\geq i}\delta_j &< \ve_i &&\text{for all }i\kin\bn\ .
  \end{align*}
  
  We will now show that for these choices of $\alpha, \deltab$ and
  $(G_i)$ the conclusion of the theorem holds.

  Let $(x_i)\ksubset S_X$ be a $\deltab$-block sequence of $(G_n)$
  with $\norm{x_i\kminus P^G_{(s_{i-1},s_i]}x_i}\kle\delta_i$ for all
  $i\kin\bn$, $1\kleq s_0\kle s_1\kle s_2\kle\dots$. Set
  \[
  z_i=\frac{P^G_{(s_{i-1},s_i]}x_i}{\norm{P^G_{(s_{i-1},s_i]}x_i}}
  \qquad\text{for all }i\kin\bn\ .
  \]
  Note that $\norm{x_i\kminus z_i}\kle 2\delta _i$ for all
  $i\kin\bn$. Replacing each $z_i$ by a small perturbation of
  itself, if necessary, we can assume that $\min\supp_G(z_i)\keq
  s_{i-1}\kplus 1$ and $\min\supp_E(z_i)\keq m_{s_{i-1}}$ for all
  $i\kin\bn$. We are going to show that for any $(a_i)\kin\coo$ we
  have
  \begin{equation}
    \label{eq:main;norm-bound}
  \Bignorm{\sum a_iz_i}\leq D \Bignorm{\sum a_i e_{m_{s_{i-1}}}}\ .
  \end{equation}
  It then follows easily from the choice of $\deltab$ that $(x_i)$ is
  $C$-dominated by $(e_{m_{s_{i-1}}})$. The proof
  of~\eqref{eq:main;norm-bound} proceeds by induction on the size of
  the support of $(a_i)$. If this is one, then the statement is
  clear. In general, we begin by choosing $z^*\kin B_{Z^*}$ such that
  \[
  \Bignorm{\sum a_iz_i}=\sum a_iz^*(z_i)\ .
  \]
  We then consider the set
  \[
  I=\{ i\kin\bn:\,\abs{z^*(z_i)}\kgeq 3\rho\}\ ,
  \]
  which splits into $I^+\keq \{ i\kin\bn:\,z^*(z_i)\kgeq 3\rho\}$ and
  $I^-\keq I\ksetminus I^+$. For a finite set $F\ksubset\bn$ we shall
  write $\ms(F)$ for the set $\{ m_{s_{i-1}}:\,i\kin F\}$. We claim
  that $\ms(I^+)$ and $\ms(I^-)$ belong to $\cF_\alpha$. Indeed, by
  the choice of $\deltab$, for any $(b_i)_{i\in I^+}\ksubset\br^+$ we
  have
  \[
  \Bignorm{\sum_{i\in I^+} b_ix_i} \geq \sum_{i\in I^+} b_iz^*(z_i) -
  \sum_{i\in I^+} b_i\kcdot 2\delta_i \geq 2\rho\sum_{i\in I^+} b_i\ .
  \]
  This shows that $(x_i)_{i\in I^+}$ belongs to $\cF$. It follows that
  $(z_i)_{i\in I^+}\kin \cG$, and $\ms(I^+)\kin \cGt\cap
  \fin{M}\ksubset \cF_\alpha$, as required. A similar argument,
  using $-z^*$ instead of $z^*$, shows that $\ms(I^-)\kin\cF_\alpha$.

  We next partition $\supp(a_i)\setminus I$ into sets
  $J_1\kle\dots\kle J_\ell$, where $\ell\kin\bn$ and we have
  \begin{equation}
    \label{eq:main;small-bit}
    \begin{array}{r@{\ }l}
    \ds 3\rho< &\bignorm{z^*|_{\spn\{z_i:\,i\in J_k\}}} \leq 6\rho \qquad
    \text{for $1\kleq k\kle \ell$, and}\\[1ex]
    & \ds\bignorm{z^*|_{\spn\{z_i:\,i\in J_\ell\}}} \leq 6\rho\ .
    \end{array}  
  \end{equation}
  This is clearly possible by the definition of $I$ and by the
  bimonotonicity of $(E_i)$. Set $F\keq \{\min J_k:\, k\keq
  1,\dots,\ell\kminus 1\}$. We claim that $\ms(F)\kin
  \cF_\alpha\ksetminus \MAX(\cF_\alpha)$, from which it follows that
  $\ms(\Ft)\kin\cF_\alpha$, where $\Ft\keq F\cup\{\min J_\ell\}$. To
  prove the claim first choose for each $k\keq 1,\dots,\ell\kminus 1$
  a vector $u_k\keq \sum_{i\in J_k} c_iz_i\kin S_Z$ such that $\sum
  c_iz^*(z_i)\kge 3\rho$. We can assume without loss of generality
  that $c_{\min J_k}\kneq 0$, \ie that $\min\supp_E(u_k)\keq
  m_{s_{\min J_k-1}}$. Set $\vt_k\keq \sum _{i\in J_k} c_ix_i$ and
  $v_k\keq \frac{\vt_k}{\norm{\vt_k}}$ for each $k\keq
  1,\dots,\ell\kminus 1$, and note that
  \[
  \norm{v_k-u_k}\leq 2\norm{\vt_k-u_k}\leq 2\sum _{i\in J_k}
  \abs{c_i}\kcdot 2\delta_i\leq 4\sum _{i\geq k}\delta_i\ .
  \]
  It follows that for any $(b_k)_{k=1}^{\ell-1}\ksubset\br^+$ we have
  \[
  \Bignorm{\sum_{k=1}^{\ell-1} b_kv_k} \geq
  \sum_{k=1}^{\ell-1} b_k z^*(u_k) - \sum_{k=1}^{\ell-1}
  b_k\kcdot \norm{v_k-u_k} \geq 2\rho \sum_{k=1}^{\ell-1} b_k\ .
  \]
  We deduce that $(v_k)\kin\cF$, $(u_k)\kin\cG$ and
  $\ms(F)\kin\cGt\cap\fin{M}\ksubset\cF_\alpha\ksetminus
  \MAX(\cF_\alpha)$, as claimed.

  The following sequence of inequalities now completes the proof
  of~\eqref{eq:main;norm-bound}.
  \begin{align*}
    \Bignorm{\sum a_iz_i} &= \sum a_iz^*(z_i) \leq \sum_{i\in
    I^+}\abs{a_i} + \sum_{i\in I^-}\abs{a_i} + \sum_{k=1}^\ell
    6\rho\kcdot\Bignorm{\sum_{i\in J_k} a_iz_i}\\[1ex]
    &\leq 2\Bignorm{\sum_{i\in I^+}
      a_ie_{m_{s_{i-1}}}}_{T_{\cF_\alpha,\frac12}} +
    2\Bignorm{\sum_{i\in I^-}
      a_ie_{m_{s_{i-1}}}}_{T_{\cF_\alpha,\frac12}}\\
    \intertext{\hfill $\ds+ 6\rho\kcdot D \sum_{k=1}^\ell
      \Bignorm{\sum_{i\in J_k}
	a_ie_{m_{s_{i-1}}}}_{T_{\cF_\alpha,\frac12}}$}\\
    &\leq (4+12\rho D)\kcdot \bignorm{\sum a_i
    e_{m_{s_{i-1}}}}_{T_{\cF_\alpha,\frac12}}\leq D\bignorm{\sum a_i
      e_{m_{s_{i-1}}}}_{T_{\cF_\alpha,\frac12}}\ .
  \end{align*}
  It is the third line where we apply the induction hypothesis. Note
  that by~\eqref{eq:main;small-bit} (and since $12\rho\kle 1$), each
  $J_k$ has size strictly smaller than that of the support of
  $(a_i)$.
\end{proof}

\section{Further remarks}

In~\cite{DF} the following universality result is proved.

\begin{thm}[\cite{DF}]
  \label{thm:dodos-ferenczi}
  For every countable ordinal $\xi$ there is a space $Y_\xi$ with
  separable dual such that every Banach space $X$ with
  $\sz(X)\kleq\xi$ embeds into $Y_\xi$.
\end{thm}

This result of P.~Dodos and V.~Ferenczi is similar to our universality
results, but the methods used are completely different. Note that
unlike Theorems~\ref{thm:universal-ca} and~\ref{thm:universal-ca2},
the above result does not give information on the Szlenk index of the
universal space $Y_\xi$. The reason for this is that the use of
descriptive set theory in proving results like
Theorem~\ref{thm:dodos-ferenczi} yields existence proofs, whereas our
approach is more constructive.

In this final section we describe the setting in which descriptive set
theory can be used to study universality problems for certain classes
of separable Banach spaces. We shall also explain what is missing if
one tries to use this approach to prove the main results of our paper.

Recall that every separable Banach space is a subspace of $C[0,1]$,
the space of continuous functions on the Cantor set. The set SB of all
closed subspaces of $C[0,1]$ is given the Effros-Borel structure,
which is the $\sigma$-algebra generated by the sets $\{
F\kin\mathrm{SB}:\, F\cap U\kneq\emptyset\}$, where $U$ ranges over
all open subsets of $C[0,1]$. This allows one to study classes of
Banach spaces according to their descriptive complexity and apply
results of
descriptive set theory. This has been first formalized by
B.~Bossard~\cite{Bossard:02}, and then taken up by S.~Argyros and
P.~Dodos~\cite{AD} to study universality problems. One of the central
notions introduced in~\cite{AD} is the following.

\begin{defn}
  A class $\cC$ of separable Banach space in SB is said to be
  \emph{strongly bounded }if for every analytic subset $\cA$ of $\cC$
  there exists $Y\kin\cC$ that contains isomorphic copies of every
  $X\kin\cA$.
\end{defn}

The main result of~\cite{DF} is that the classes $\cS\cR$ of
separable, reflexive spaces and $\cS\cD$ of spaces with separable dual
are strongly bounded. Since $\{ X\kin\cS\cD:\,\sz(X)\kleq\xi\}$ is
analytic (even Borel, this was proved in~\cite{Bossard:02}),
Theorem~\ref{thm:dodos-ferenczi} follows. However, it was not known
whether the classes $\cC_\alpha$ from Pe\l czy\'nski's question were
analytic or not, and so the main theorem from~\cite{DF} could not be
applied. From our results we can now prove the following.

\begin{thm}
  \label{thm:ca-analytic}
  For every countable ordinal $\alpha$ the class $\cC_\alpha$ is
  analytic in the Effros-Borel structure of SB.
\end{thm}

\begin{proof}
  Fix a countable ordinal $\alpha$. We begin by showing that the class
  $\cC_{\omega^{\alpha\cdot\omega}}$ is analytic. By
  Theorem~\ref{thm:universal-ca2} if
  $X\kin\cC_{\omega^{\alpha\cdot\omega}}$, then there exists
  $n\kin\bn$ such that $X$ isomorphically embeds into $Z_{\alpha,n}$
  which we denote by $X\khookrightarrow Z_{\alpha,n}$. Conversely,
  assume that $X\khookrightarrow Z_{\alpha,n}$. Since $Z_{\alpha,n}$
  has an FDD satisfying subsequential $(\di{T}{*}{\alpha\cdot
  n},T_{\alpha\cdot n})$ estimates, it follows easily that $X$
  satisfies subsequential $(\di{T}{*}{\alpha\cdot n},T_{\alpha\cdot
  n})$-tree estimates. By duality the same holds for $X^*$, and hence
  $X^*$ also embeds into $Z_{\alpha,n}$. From
  Proposition~\ref{prop:upper-tsirelson-upper-szlenk} we now obtain
  \[
  \max\{ \sz(X),\sz(X^*)\} \leq \sz(Z_{\alpha\cdot n}) \leq
  \omega^{\alpha\cdot n}\ ,
  \]
  and so $X\kin\cC_{\omega^{\alpha\cdot\omega}}$.
  
  It is well known and easy to show that for any $Y\kin\mathrm{SB}$
  the set $\{X\kin\mathrm{SB}:\,X\khookrightarrow Y\}$ is analytic. It
  follows that
  \[
  \cC_{\omega^{\alpha\cdot\omega}} = \bigcup _{n\in\bn} \{
  X\kin\mathrm{SB}:\, X\khookrightarrow Z_{\alpha,n}\}
  \]
  is analytic, as claimed.
  
  To prove the general case, we use a recent result of
  P.~Dodos~\cite{Dodos} which states that
  \[
  \cS_\alpha = \{ X\kin\mathrm{SB}:\,\max\{\sz(X),\sz(X^*)\}\kleq\alpha
  \}
  \]
  is analytic.
  Since $\cC_\alpha\keq \cS_\alpha\cap
  \cC_{\omega^{\alpha\cdot\omega}}$, it follows immediately that
  $\cC_\alpha$ is also analytic.
\end{proof}

\begin{rem}
  As mentioned in the Introduction,it was C.~Rosendal who pointed out
  to us that the analyticity of $\cC_{\omega^{\alpha\cdot\omega}}$
  follows from our results. Later P.~Dodos informed us that this fact
  together with his result implies the general case.
\end{rem}

\input{szlenk.bbl}

\vspace{4ex}

\noindent
\parbox[t]{0.5\textwidth}{%
  E. Odell\\
Department of Mathematics,\\
The University of Texas,\\
1 University Station C1200,\\
Austin, TX 78712, USA\\
\texttt{odell@math.utexas.edu}}
\parbox[t]{0.5\textwidth}{%
Th. Schlumprecht\\
Department of Mathematics,\\
Texas A\&M University,\\
College Station, TX 78712\\
\texttt{schlump@math.tamu.edu}}

\vspace{3ex}

\noindent
\parbox[t]{\textwidth}{%
A. Zs\'ak\\
School of Mathematical Sciences\\
University of Nottingham\\
University Park\\
Nottingham NG7 2RD, England\\
\texttt{andras.zsak@maths.nottingham.ac.uk}}

\end{document}

%% file: mypreamble.tex
\usepackage{ifthen}

\newcounter{mylisti} \newcounter{mylistii}
\newcounter{nest}
\newcommand{\defaultlabel}{}

\newenvironment{mylist}[1]{%
  \addtocounter{nest}{1}
  \ifthenelse{\value{nest}=1}{%
    \renewcommand{\defaultlabel}{(\roman{mylisti})\hfill}}{%
    \renewcommand{\defaultlabel}{(\alph{mylistii})\hfill}}
  \begin{list}{\defaultlabel}{%
      \ifthenelse{\value{nest}=1}{\usecounter{mylisti}}{%
        \usecounter{mylistii}}
      
      \addtolength{\topsep}{1ex}
      \addtolength{\itemsep}{0.5ex}
      \settowidth{\labelwidth}{#1}
      \setlength{\leftmargin}{\labelwidth}
      \addtolength{\leftmargin}{\labelsep}}}{\addtocounter{nest}{-1}
\end{list}}

\newcommand{\bn}{\ensuremath{\mathbb N}}

\newcommand{\br}{\ensuremath{\mathbb R}}

\newcommand{\vx}{\ensuremath{\boldsymbol{x}}}
\newcommand{\vy}{\ensuremath{\boldsymbol{y}}}

\newcommand{\cA}{\ensuremath{\mathcal A}}

\newcommand{\cC}{\ensuremath{\mathcal C}}
\newcommand{\cD}{\ensuremath{\mathcal D}}
\newcommand{\cF}{\ensuremath{\mathcal F}}
\newcommand{\cFt}{\ensuremath{\tilde{\cF}}}
\newcommand{\cG}{\ensuremath{\mathcal G}}
\newcommand{\cGt}{\ensuremath{\tilde{\cG}}}
\newcommand{\cH}{\ensuremath{\mathcal H}}

\newcommand{\cR}{\ensuremath{\mathcal R}}
\newcommand{\cS}{\ensuremath{\mathcal S}}

\newcommand{\vt}{\ensuremath{\tilde{v}}}

\newcommand{\xt}{\ensuremath{\tilde{x}}}

\newcommand{\Ft}{\ensuremath{\tilde{F}}}

\newcommand{\Cb}{\ensuremath{\overline{C}}}

\newcommand{\deltab}{\ensuremath{\bar{\delta}}}

\newcommand{\veb}{\ensuremath{\bar{\ve}}}
\newcommand{\Us}{\ensuremath{U^{(*)}}}

\newcommand{\Zs}{\ensuremath{Z^{(*)}}}

\newcommand{\blder}[1]{\ensuremath{^{(#1)}_{\mathrm{bl}}}} 
\newcommand{\bli}{\ensuremath{\mathrm{I}_{\mathrm{bl}}}} 
\newcommand{\cbder}[1]{\ensuremath{^{(#1)}_{\mathrm{CB}}}} 
\newcommand{\cbi}{\ensuremath{\mathrm{I}_{\mathrm{CB}}}} 
\newcommand{\cof}{\ensuremath{\mathrm{cof}}}

\newcommand{\MAX}{\ensuremath{\mathrm{MAX}}}

\newcommand{\spn}{\ensuremath{\mathrm{span}}}
\newcommand{\supp}{\ensuremath{\mathrm{supp}}}
\newcommand{\sz}{\ensuremath{\mathrm{Sz}}} 

\newcommand{\fin}[1]{{[#1]}^{<\omega}}

\newcommand{\infin}[1]{{[#1]}^{\omega}}
\newcommand{\wder}[1]{\ensuremath{^{(#1)}_{\mathrm{w}}}} 
\newcommand{\wi}{\ensuremath{\mathrm{I}_{\mathrm{w}}}} 

\newcommand{\kcdot}{\!\cdot\!}
\newcommand{\keq}{\!=\!}

\newcommand{\kin}{\!\in\!}
\newcommand{\kge}{\!>\!}
\newcommand{\kgeq}{\!\geq\!}

\newcommand{\khookrightarrow}{\!\hookrightarrow\!}
\newcommand{\kle}{\!<\!}
\newcommand{\kleq}{\!\leq\!}

\newcommand{\kminus}{\!-\!}
\newcommand{\kneq}{\!\neq\!}

\newcommand{\kplus}{\!+\!}

\newcommand{\ksetminus}{\!\setminus\!}

\newcommand{\ksubset}{\!\subset\!}

\newcommand{\ksupset}{\!\supset\!}

\newcommand{\ktimes}{\!\times\!}

\newcommand{\abs}[1]{\lvert #1\rvert}

\newcommand{\norm}[1]{\lVert #1\rVert}
\newcommand{\bignorm}[1]{\big\lVert #1\big\rVert}
\newcommand{\Bignorm}[1]{\Big\lVert #1\Big\rVert}

\newcommand{\tnorm}[1]{\lvert\mspace{-1mu}\lvert\mspace{-1mu}\lvert
  #1\rvert\mspace{-1mu}\rvert\mspace{-1mu}\rvert}

\newcommand{\co}{\mathrm{c}_0}
\newcommand{\coo}{\mathrm{c}_{00}}

\newcommand{\V}{\forall \,}
\newcommand{\E}{\exists \,}
\newcommand{\ve}{\varepsilon}

\newcommand{\di}[3]{\text{\makebox[0pt][l]{${#1}^{#2}$}}{#1}_{#3}}
\newcommand{\ds}{\displaystyle}

\newcommand{\ie}{\textit{i.e.,}\ }

%% file: szlenk.bbl
\def\cprime{$'$}